\documentclass[
a4paper, 
12pt, 
BCOR=6mm,
fleqn, 
twoside, 
headsepline, 
parskip=half,
]{scrartcl}

\usepackage[english]{babel}
\usepackage[utf8]{inputenc}
\usepackage[T1]{fontenc}
\usepackage[english,cleanlook]{isodate}
\usepackage{lmodern}
\usepackage[babel=true]{microtype}
\usepackage{amsmath}
\usepackage{mathtools}
\usepackage{amssymb}
\usepackage{amsthm}
\usepackage{scrlayer-scrpage} 
\usepackage{enumitem}
\usepackage{tikz} 
\usetikzlibrary{calc}             

\ohead{\footnotesize Mario B. Schulz}
\ihead{\footnotesize Instantaneously complete Yamabe flow on hyperbolic space}
\setkomafont{pageheadfoot}{}

\providecommand{\N}{\mathbb{N}}    
\providecommand{\R}{\mathbb{R}}

\providecommand{\Hyp}{\mathbb{H}}
\providecommand{\Eucl}{\mathbb{E}}
\providecommand{\Sp}{\mathbb{S}}
\providecommand{\Mfd}{M}
\providecommand{\Bl}{B} 
\providecommand{\Zl}{Z}                          
\providecommand{\dd}{\mathrm{d}}
\providecommand{\ex}{\mathrm{e}}
\providecommand{\Ct}{\mathrm{C}}
\providecommand{\Rsc}{\mathrm{R}} 
\DeclarePairedDelimiter\abs{\lvert}{\rvert}
\DeclarePairedDelimiter\nm{\lVert}{\rVert}
\DeclarePairedDelimiter\sk{\langle}{\rangle}
\DeclarePairedDelimiter\interval{]}{[}
\DeclarePairedDelimiter\Interval{[}{[}
\DeclarePairedDelimiter\intervaL{]}{]}
\DeclareMathOperator{\Ric}{Ric}
 
\DeclareMathOperator{\rel}{v}

\theoremstyle{plain}
\newtheorem{thm}{Theorem} 
\newtheorem{lem}{Lemma}[section]
\newtheorem{cor}[lem]{Corollary} 
\newtheorem{prop}[lem]{Proposition}

\theoremstyle{definition}
\newtheorem*{defn*}{Definition}
 
\theoremstyle{remark}
\newtheorem*{rem*}{Remark}
 
\usepackage[
pdftitle={Instantaneously complete Yamabe flow on hyperbolic space},
pdfauthor={Mario B. Schulz}, 
pdfborder={0 0 0},
]{hyperref}

\title{Instantaneously complete Yamabe~flow on hyperbolic space}
\author{Mario B. Schulz\thanks{This research was supported by the Swiss National Science Foundation  under grant 200020\_159925.}
\\[-.5ex]\footnotesize\itshape
Queen Mary University of London, 
Mile End Road, London E1 4NS 
}
\date{\today}
 
\begin{document}

\maketitle

\begin{abstract}
We prove existence of instantaneously complete Yamabe flows on hyperbolic space of arbitrary dimension $m\geq3$. 
The initial metric is assumed to be conformally hyperbolic with conformal factor and scalar curvature bounded from above. 
We do not require initial completeness or bounds on the Ricci curvature. 
If the initial data are rotationally symmetric, the solution is proven to be unique in the class of instantaneously complete, rotationally symmetric Yamabe flows.
\end{abstract}

%===== INTRODUCTION ====================================

The Yamabe flow was introduced by Richard Hamilton \cite{Hamilton1989}. 
It describes a family of Riemannian metrics $g(t)$ subject to the equation $\partial_{t}g=-\Rsc g$ and tends to evolve a given initial metric towards a metric of vanishing scalar curvature.  
Hamilton showed that global solutions always exist on compact manifolds without boundary. 
Their asymptotic behaviour was subsequently analysed by Chow \cite{Chow1992}, Ye \cite{Ye1994}, Schwetlick and Struwe \cite{Schwetlick2003} and Brendle \cite{Brendle2005, Brendle2007}. 
Less is known about the Yamabe flow on noncompact manifolds.    
Daskalopoulos and Sesum \cite{Daskalopoulos2013} analysed the profiles of self-similar solutions (Yamabe solitons). 
Ma and An \cite{Ma1999} proved short-time existence of Yamabe flows on noncompact, locally conformally flat manifolds $\Mfd$ under the assumption that the initial manifold $(\Mfd,g_0)$ is complete with Ricci tensor bounded from below. 
More recently, Bahuaud and Vertman \cite{Bahuaud2014,Bahuaud2016} constructed Yamabe flows starting from spaces with incomplete edge singularities such that the singular structure is preserved along the flow.

In dimension $m=2$ the Yamabe flow coincides with the Ricci flow. 
Peter Topping and Gregor Giesen \cite{Topping2010,Giesen2011,Topping2015} introduced the notion of instantaneous completeness and obtained existence and uniqueness of instantaneously complete Ricci/Yamabe flows on arbitrary surfaces. 
The analysis of the flow on the hyperbolic disc plays an important role in their work.  
It relies on results which exploit the fact that the Ricci tensor is bounded by the scalar curvature in dimension 2. 
 
The goal of this paper is to find techniques which allow a generalisation of Giesen and Topping's results to the Yamabe flow on hyperbolic space $(\Hyp,g_{\Hyp})$ of dimension $m\geq3$.   
$(\Hyp,g_{\Hyp})$ is a complete, noncompact, simply connected manifold of constant sectional curvature $-1$ and it is conformally equivalent to the Euclidean unit ball $(\Bl_1,g_{\Eucl})$.  

\begin{defn*}
A family $(g(t))_{t\in[0,T]}$ of Riemannian metrics on a manifold $\Mfd$ with scalar curvature $\Rsc=\Rsc_{g(t)}$ is called a \emph{Yamabe flow}, if $\tfrac{\partial}{\partial t}g=-\Rsc\, g$. 
The family $(g(t))_{t\in[0,T]}$ is called \emph{instantaneously complete}, 
if the Riemannian manifold $(\Mfd,g(t))$ is geodesically complete for every $0<t\leq T$. 
\end{defn*}

Since the Yamabe flow preserves the conformal class of the metric, any conformally hyperbolic Yamabe flow $(g(t))_{t\in[0,T]}$ on $\Hyp$ is given by $g(t)=u(\cdot,t)\,g_{\Hyp}$, 
where the conformal factor $u\colon\Hyp\times[0,T]\to\R$ is a positive function evolving by the equation 
\begin{align}\label{eqn:Yamabe-flow}
\frac{1}{m-1}\frac{\partial u}{\partial t}
 =-\frac{u\,\Rsc}{m-1} 
&=m+\frac{\Delta_{g_{\Hyp}} u}{u}
+\frac{(m-6)}{4}\frac{\abs{\nabla u}_{g_{\Hyp}}^2}{u^2}
\end{align}
where $m=\dim\Hyp$, where $\Delta_{g_{\Hyp}}$ denotes the Laplace-Beltrami operator with respect to the hyperbolic background metric $g_\Hyp$ and where $\abs{\nabla u}_{g_{\Hyp}}^2=g_{\Hyp}(\nabla u,\nabla u)$. 
Introducing the exponent $\eta\vcentcolon=\frac{m-2}{4}$ to define $U=u^{\eta}$, equation \eqref{eqn:Yamabe-flow} is equivalent to  
\begin{align}\label{eqn:Yamabe-flow-eta}
\frac{U^{\frac{1}{\eta}}}{m-1}\frac{\partial U}{\partial t}
&=m\eta U+\Delta_{g_{\Hyp}} U
\end{align}
which follows by virtue of $(\eta-1)=\frac{(m-6)}{4}$ and $\tfrac{1}{\eta}\Delta_{g_{\Hyp}} u^\eta=(\eta-1)u^{\eta-2}\abs{\nabla u}_{g_{\Hyp}}^2+u^{\eta-1}\Delta_{g_{\Hyp}} u$. 
While equation \eqref{eqn:Yamabe-flow-eta} has a simpler structure, pointwise bounds on $u$ follow easier from equation \eqref{eqn:Yamabe-flow}. 
We prove the following statements.

\begin{thm}[Existence]\label{thm:existence}
Let $g_0=u_0g_{\Hyp}$ be any (possibly incomplete) conformal metric on $(\Hyp,g_{\Hyp})$ with bounded conformal factor $0<u_0\in\Ct^{4,\alpha}(\Hyp)$ and scalar curvature $\Rsc_{g_0}\leq K_0$ bounded from above.  
Then, for any $T>0$ there exists an instantaneously complete family of metrics $(g(t))_{t\in[0,T]}$ satisfying the Yamabe flow equation 
\begin{align*}
\left\{\begin{aligned}
\tfrac{\partial}{\partial t}g(t)&=-\Rsc_{g(t)}\,g(t) \quad&&\text{ in }\Hyp\times[0,T], \\[1ex]
g(0)&=g_0 \quad&&\text{ on }\Hyp.
\end{aligned}\right.
\end{align*}
Moreover, $g(t)\geq m(m-1)\,t\,g_{\Hyp}$ for any $t\in\intervaL{0,T}$. 
As $t\searrow0$, the metric $g(t)$ converges locally in class $\Ct^2$ to $g_0$. 
\end{thm}

\begin{rem*}
On noncompact, locally conformally flat manifolds $\Mfd$, Ma and An \cite{Ma1999} require bounded scalar curvature, a lower bound on the Ricci tensor and completeness of $(\Mfd,g_0)$ for short-time existence and additionally non-positive scalar curvature for global existence. 
\end{rem*}

\begin{thm}[Uniqueness]\label{thm:uniqueness}
Let $(g(t))_{t\in[0,T]}$ and $(\tilde{g}(t))_{t\in[0,T]}$ be two conformally hyperbolic Yamabe flows on $(\Hyp,g_{\Hyp})$ satisfying 
\begin{enumerate}[label={\normalfont(\roman*)}]
\item\label{cond:Uniqueness-i}
$\exists b\in\R: \quad \tilde{g}(0)\leq b\,g_{\Eucl}$, 
\item\label{cond:Uniqueness-ii}
$\forall t\in[0,T]:\quad g(t)\geq m(m-1)t\,g_{\Hyp}$.
\end{enumerate}
Then, if $\tilde{g}(0)\leq g(0)$, we have $\tilde{g}(t)\leq g(t)$ for all $t\in[0,T]$. 

In particular, if $g(t)$ and $\tilde{g}(t)$ both satisfy \ref{cond:Uniqueness-ii} and if $g(0)=\tilde{g}(0)\leq b\,g_{\Eucl}$, then $g\equiv\tilde{g}$. 
\end{thm}

\begin{rem*}
In the Poincar\'{e} ball model for $\Hyp$, the conformally hyperbolic initial metric $\tilde{g}(0)$ can be compared to the Euclidean metric, whose pullback we also denote as $g_{\Eucl}$. 
Assumption \ref{cond:Uniqueness-i} means, that the initial manifold $(\Hyp,\tilde{g}(0))$ is incomplete and has finite diameter. 

Assumption \ref{cond:Uniqueness-ii} implies instantaneous completeness of $g(t)$. 
We conjecture that instantaneously complete, conformally hyperbolic Yamabe flows always satisfy~\ref{cond:Uniqueness-ii}. 
For rotationally symmetric flows, this is proved in Proposition \ref{prop:lower}. 

The instantaneously complete flow Topping \cite{Topping2010} constructs on $2$-dimensional manifolds has a certain maximality property which we also observe in higher dimensions: 
Theorem~\ref{thm:uniqueness} implies, that if $g_0\leq b\,g_{\Eucl}$, then the Yamabe flow constructed in Theorem~\ref{thm:existence} is \emph{maximally stretched} in the sense that any other Yamabe flow with the same or lower initial data stays below it. 

Moreover, Theorem~\ref{thm:uniqueness} implies, that if $g_0\leq b\,g_{\Eucl}$, then any two solutions $(g(t))_{t\in[0,T]}$ and $(\tilde{g}(t))_{t\in[0,\tilde{T}]}$ constructed in Theorem~\ref{thm:existence} agree on $[0,T]\cap[0,\tilde{T}]$. 
Since $T>0$ is arbirtary in Theorem~\ref{thm:existence}, we then obtain global existence, i.\,e. an instantaneously complete Yamabe flow $(g(t))_{t\in\Interval{0,\infty}}$ on $\Hyp$ with $g(0)=g_0$. 
\end{rem*}

\begin{thm}\label{thm:nonexistence}
Let $g_0=u_0g_{\R^m}$ be a conformally Euclidean metric on $(\R^m,g_{\R^m})$ with $m\geq 3$. 
If  $u_0(x)\leq b\abs{x}^{-4}$ for some finite constant $b$ and all $x\in\R^m$, then any Yamabe flow $(g(t))_{t\in[0,T]}$ on $\R^m$ with $g(0)=g_0$ is geodesically incomplete for all $t\in[0,T]$.
\end{thm}

\begin{rem*}
Theorem \ref{thm:nonexistence} shows that the results about instantaneously complete Yamabe flow on hyperbolic space do not equally hold on arbitrary manifolds of dimension $m\geq3$. 
It contrasts with the $2$-dimensional case, where instantaneously complete Yamabe flows always exist \cite{Giesen2011}.  

For example, there does not exist an instantaneously complete Yamabe flow starting from the punctured unit sphere $(\dot{\Sp}^{m},g_{\Sp^{m}})$ in dimension $m\geq3$. 
Indeed, if $\pi\colon\dot{\Sp}^{m}\to\R^m$ is stereographic projection, then \(\pi_*g_{\Sp^{m}}=4{(1+\abs{x}^2)}^{-2}g_{\R^m}\) and Theorem \ref{thm:nonexistence} applies. 
\end{rem*}

\paragraph{Acknowledgements.}
The author is very grateful to the anonymous referee for his work and care including valuable comments and suggestions.

%===== EXISTENCE ========================================

\section{Existence}

In this section, we prove Theorem \ref{thm:existence}. 
As a first step, short-time existence of a solution $u$ to equation \eqref{eqn:Yamabe-flow} for given $u(\cdot,0)=u_0>0$ on convex, bounded domains $\Omega\subset\Hyp$ with suitable boundary data is proven by applying the inverse function theorem on Banach spaces. 
Richard Hamilton \cite[\textsection\,IV.11]{Hamilton1975} uses the same technique to prove existence of solutions to the heat equation for manifolds. 
Local H\"older estimates then lead to a uniform existence time for all domains. 

In a second step we derive uniform gradient estimates, which do not depend on the domain. 
By considering an exhaustion of $\Hyp$ with convex, bounded domains, we obtain a locally uniformly bounded sequence of solutions which allows a subsequence converging to a solution of \eqref{eqn:Yamabe-flow} on all of $\Hyp$.

\subsection{Existence on bounded domains}

We denote the non-linear terms in equation \eqref{eqn:Yamabe-flow} by 
\begin{align*}
Q[u]\vcentcolon=(m-1)\Bigl(m+\frac{\Delta_{g_{\Hyp}} u}{u}+\frac{(m-6)}{4}\frac{\abs{\nabla u}_{g_{\Hyp}}^2}{u^2}\Bigr).
\end{align*}
Given a smooth, bounded domain $\Omega\subset\Hyp$ and $T>0$, we consider the problem
\begin{align}\label{eqn:pde}
\left\{
\begin{aligned}
\frac{\partial u}{\partial t}&=Q[u] 
&&\text{ in $\Omega\times[0,T]$, } \\
u&=\phi
&&\text{ on $\partial\Omega\times[0,T]$, }  \\[.5ex]
u&=u_0 
&&\text{ on $\Omega\times\{0\}$} 
\end{aligned}\right.
\end{align}
for given $0<u_0\in\Ct^{2,\alpha}(\overline{\Omega})$ and $\phi\in\Ct^{2,\alpha;1,\frac{\alpha}{2}}(\partial\Omega\times[0,T])$ satisfying $\phi(\cdot,t)\geq m(m-1)t$ and the first order compatibility conditions
\begin{align}\label{eqn:compatibility}
\left\{
\begin{aligned}
\phi(\cdot,0)&=u_0 
&&\text{ on $\partial\Omega$, }\\
\frac{\partial\phi}{\partial t}(\cdot,0)&=Q[u_0]
=-u_0\Rsc_{g_0}
&&\text{ on $\partial\Omega$. }
\end{aligned}\right.
\end{align} 
Such boundary data $\phi$ exist since $u_0$ and $\Rsc_{g_0}$ are bounded on the compact set $\partial\Omega$ and $u_0>0$. 
In section \ref{sec:localestimates} we choose $\phi$ explicitly. 

For small times $t>0$, we expect the solution $u$ to \eqref{eqn:pde} to be close to the solution $\tilde{u}$ of the linear problem  
\begin{align}\label{eqn:pde-linear-u}
\left\{\begin{aligned}
\frac{1}{m-1}\frac{\partial\tilde{u}}{\partial t}-\frac{\Delta_{g_{\Hyp}}\tilde{u}}{u_0}
-\frac{(m-6)}{4}\frac{\sk{\nabla\tilde{u},\nabla u_0}_{g_{\Hyp}}}{u_0^2}
&=m
&&\text{ in $\Omega\times[0,T]$, } \\
\tilde{u}&=\phi
&&\text{ on $\partial\Omega\times[0,T]$, }  \\[.5ex]
\tilde{u}&=u_0
&&\text{ on $\Omega\times\{0\}$. } 
\end{aligned}\right.
\end{align}
Since $\Omega$ is bounded and since $u_0>0$ in $\Hyp$, 
there exists some $\delta>0$ depending on $\Omega$ and $u_0$ such that $u_0\geq \delta$ in $\Omega$. 
Therefore, equation \eqref{eqn:pde-linear-u} is uniformly parabolic with regular coefficients and the compatibility conditions given in \eqref{eqn:compatibility} are satisfied.  
According to linear parabolic theory \cite[\textsection\,IV.5, Theorem 5.2]{Ladyzenskaja1967}, problem \eqref{eqn:pde-linear-u} has a unique solution $\tilde{u}\in\Ct^{2,\alpha;1,\frac{\alpha}{2}}(\overline{\Omega}\times[0,T])$. 
Since $u_0>0$ and $\phi(\cdot,t)\geq m(m-1)t$ for all $t\in[0,T]$, the parabolic maximum principle (Prop. \ref{prop:parabolicMP}) applied to $m(m-1)t-\tilde{u}(\cdot,t)$ implies 
\begin{align*}
\tilde{u}(\cdot,t)\geq m(m-1)t. 
\end{align*} 
In particular, $\tilde{u}\geq\varepsilon$ on $\Omega\times[0,T]$ for some $\varepsilon>0$ depending on $\Omega$ and $\tilde{u}$.

\begin{lem}[Short-time existence on bounded domains]
\label{lem:shorttimeexistence}
Let $\Omega\subset\Hyp$ be a smooth, convex, bounded domain. 
Then, there exists $T>0$ such that problem \eqref{eqn:pde} is solvable. 
\end{lem}

\begin{proof}
A solution $u$ to \eqref{eqn:pde} is of the form $u=\tilde{u}+v$, where $\tilde{u}$ solves \eqref{eqn:pde-linear-u} and
\begin{align}\label{eqn:pde-v}
\left\{\begin{aligned}
\frac{\partial v}{\partial t}
&=Q[\tilde{u}+v]-\frac{\partial \tilde{u}}{\partial t}
&&\text{ in $\Omega\times[0,T]$, } \\
v&=0
&&\text{ on $\partial\Omega\times[0,T]$, }  \\[.5ex]
v&=0
&&\text{ on $\Omega\times\{0\}$. } 
\end{aligned}\right.
\end{align}
Given the H\"older exponent $0<\alpha<1$, let 
\begin{alignat*}{3}
X&\vcentcolon=\{&v&\in\Ct^{2,\alpha;1,\frac{\alpha}{2}}(\overline{\Omega}\times[0,T])
\mid{}&v&=0\text{ on }(\Omega\times\{0\})\cup(\partial\Omega\times[0,T])\},
\\[1ex]
Y&\vcentcolon=\{&f&\in
\Ct^{0,\alpha;0,\frac{\alpha}{2}}(\overline{\Omega}\times[0,T]) 
\mid{}&f&=0\text{ on }\partial\Omega\times\{0\}\}.  
\end{alignat*}
The map 
\begin{align*}
S\colon X&\to Y
\\
v&\mapsto \tfrac{\partial}{\partial t}(\tilde{u}+v)-Q[\tilde{u}+v]. 
\end{align*}
then is well-defined because the compatibility conditions \eqref{eqn:compatibility} imply that at every $p\in\partial\Omega$ for every $v\in X$, we have 
\begin{align*}
(S v)(p,0)&=\Bigl(\frac{\partial\tilde{u}}{\partial t}-Q[\tilde{u}]\Bigr)(p,0)
=\Bigl(\frac{\partial\phi}{\partial t}(\cdot,0)-Q[u_0]\Bigr)(p)=0.
\end{align*}
The linearisation of $Q[\tilde{u}]$ around $\tilde{u}\in\Ct^{2,\alpha;1,\frac{\alpha}{2}}(\overline{\Omega}\times[0,T])$ defines the linear operator 
\begin{align*}
L(\tilde{u})&=(m-1)\Bigl(-\frac{\Delta_{g_{\Hyp}}\tilde{u}}{\tilde{u}^2}
-\frac{(m-6)}{2}\frac{\abs{\nabla\tilde{u}}_{g_{\Hyp}}^2}{\tilde{u}^3}+
\frac{(m-6)}{2\tilde{u}^2}\sk{\nabla\tilde{u},\nabla\,\cdot\,}_{g_{\Hyp}}
+\frac{\Delta_{g_{\Hyp}}}{\tilde{u}}
\Bigr). 
\end{align*}

\pagebreak[3]

The map $S$ is Gateaux differentiable at $0\in X$ with derivative
\begin{align*}
D S(0)\colon X&\to Y
\\
w&\mapsto \tfrac{\partial}{\partial t}w-L(\tilde{u})w.
\end{align*}
The mapping $u\mapsto L(u)$ is continuous near $\tilde{u}$ because $\tilde{u}$ is bounded away from zero. 
Hence, $DS(0)$ is in fact the Fr\'{e}chet-derivative of $S$ at $0\in X$. 
Moreover, the linear operator $\frac{\partial}{\partial t}-L(\tilde{u})$ is uniformly parabolic. 

Let $f\in Y$ be arbitrary. 
By definition, $0=f(\cdot,0)$ is satisfied on $\partial\Omega$ which is the first order compatibility condition for the linear parabolic problem 
\begin{align}\label{eqn:pde-linear}
\left\{\begin{aligned}
\frac{\partial w}{\partial t}-L(\tilde{u})w&=f
&&\text{ in $\Omega\times[0,T]$, } \\
w&=0 
&&\text{ on $\partial\Omega\times[0,T]$, } \\[.5ex]
w&=0
&&\text{ on $\Omega\times\{0\}$. } 
\end{aligned}\right.
\end{align}
As before, linear parabolic theory states that \eqref{eqn:pde-linear} has a unique solution $w\in X$. 
Therefore, the linear map $DS(0)\colon X\to Y$ is invertible. 

By the Inverse Function Theorem (Proposition \ref{prop:IFT}), $S$ is invertible in some neighbourhood $V\subset Y$ of $S(0)$. 
We claim that $V$ contains an element $e$ such that $e(\cdot,t)=0$ for $0\leq t\leq\varepsilon$ and sufficiently small $\varepsilon>0$. 
Let $f\vcentcolon=S(0)=\frac{\partial}{\partial t}\tilde{u}-Q[\tilde{u}]$ be fixed. 
Let $\theta\colon[0,T]\to[0,1]$ be a smooth cutoff function such that
\begin{align*}
\theta(t)&=\begin{cases}
0, &\text{ for } t\leq\varepsilon, \\ 
1, &\text{ for } t>2\varepsilon,
\end{cases}&
0&\leq\frac{\dd\theta}{\dd t}\leq\frac{3}{\varepsilon}.
\end{align*}
We claim $\theta f\in V$ for sufficiently small $\varepsilon>0$. 
Since $\tilde{u}$ is smooth in $\overline{\Omega}\times[0,T]$, we have $f\in\Ct^{1}(\overline{\Omega}\times[0,T])$. 
Since at $t=0$, we have
\begin{align}\label{est:f-hoelder1}
f(\cdot,0)&=\frac{\partial\tilde{u}}{\partial t}(\cdot,0)-Q[u_0]=0 \quad\text{ on $\overline{\Omega}$, }   
\end{align}
we can estimate 
\begin{align}\label{eqn:sC1}
\abs{f(\cdot,s)}=\abs{f(\cdot,s)-f(\cdot,0)}\leq s\nm{f}_{\Ct^1(\overline{\Omega}\times[0,T])}.
\end{align}
Let $t,s\in[0,T]$ such that $t>s$. 
If $s>2\varepsilon$, then $(f-\theta f)(\cdot,s)=(f-\theta f)(\cdot,t)=0$. 
Therefore, we may assume $s\leq2\varepsilon$. 
In this case we estimate
\begin{align}\notag
&\abs{(f-\theta f)(\cdot,t)-(f-\theta f)(\cdot,s)}
\\[.2ex]\notag
&\leq\abs{f(\cdot,t)-f(\cdot,s)}
+\abs{\theta f(\cdot,t)-\theta f(\cdot,s)}
\\[.2ex]\notag
&\leq
\bigl(1+\abs{\theta(t)}\bigr)\abs{f(\cdot,t)-f(\cdot,s)}
+\abs{f(\cdot,s)}\abs{\theta(t)-\theta(s)}
\\[.2ex]\notag
&\leq2\nm{f}_{\Ct^1}\abs{t-s}+s\nm{f}_{\Ct^1}\nm{\theta'}_{\Ct^0}\abs{t-s}
\\[.2ex]\notag
&\leq\bigl(2+s\tfrac{3}{\varepsilon}\bigr)\nm{f}_{\Ct^1}\abs{t-s}
\\[.2ex]\label{est:f-hoelder3}
&\leq8\nm{f}_{\Ct^1}\abs{t-s}.
\end{align}
Due to \eqref{est:f-hoelder1}, the special case $s=0$ reduces to 
\begin{align}\label{eqn:20190125-3}
\abs{(f-\theta f)(\cdot,t)}
&\leq8t\nm{f}_{\Ct^1}. 
\end{align}
Since the left-hand side of \eqref{eqn:20190125-3} vanishes for $t>2\varepsilon$, we have in fact  
\begin{align*} 
\nm{f-\theta f}_{\Ct^0}
&\leq16\varepsilon\nm{f}_{\Ct^1}. 
\end{align*}
If $\abs{t-s}<\varepsilon$, estimate \eqref{est:f-hoelder3} directly implies 
\begin{align*}
\abs{(f-\theta f)(\cdot,t)-(f-\theta f)(\cdot,s)}
&\leq 8\varepsilon^{1-\frac{\alpha}{2}}\nm{f}_{\Ct^1}\abs{t-s}^{\frac{\alpha}{2}}. 
\end{align*}
If $\abs{t-s}\geq\varepsilon$, we replace the estimate by 
\begin{align*}
\abs{(f-\theta f)(\cdot,t)-(f-\theta f)(\cdot,s)}
&\leq2\nm{f-\theta f}_{\Ct^0}
\\[.75ex]
&\leq32\varepsilon \nm{f}_{\Ct^1}
\\[.75ex]
&\leq32\varepsilon^{1-\frac{\alpha}{2}}\nm{f}_{\Ct^1}\abs{t-s}^{\frac{\alpha}{2}}.
\end{align*}
Therefore, $[f-\theta f]_{\frac{\alpha}{2},t}\leq32\varepsilon^{1-\frac{\alpha}{2}}\nm{f}_{\Ct^1}$. 

For the spatial H\"older seminorm, we obtain a similar estimate from \eqref{eqn:sC1} and 
\begin{align*}
&\abs{(f-\theta f)(x,t)-(f-\theta f)(y,t)}
\\
&\leq \abs{1-\theta(t)}\abs{f(x,t)-f(y,t)}^{\alpha}\abs{f(x,t)-f(y,t)}^{1-\alpha}
\\
&\leq \nm{f}_{\Ct^1}^{\alpha}\,d(x,y)^{\alpha}\bigl(4\varepsilon\nm{f}_{\Ct^1}\bigr)^{1-\alpha}
=(4\varepsilon)^{1-\alpha}\nm{f}_{\Ct^1}\,d(x,y)^{\alpha},
\end{align*}
where $d(x,y)$ denotes the Riemannian distance between $x$ and $y$ in $(\Hyp,g_{\Hyp})$ and where convexity of $\Omega$ is used. 
To conclude, $\nm{f-\theta f}_{Y}\leq C\varepsilon^{\beta-\alpha}\nm{f}_{\Ct^1}$. 
Thus, $\theta f$ belongs to the neighbourhood $V$ of $f$ if  $\varepsilon>0$ is sufficiently small. 
By construction, $S^{-1}(\theta f)$ is a solution to \eqref{eqn:pde-v} in $\Omega\times[0,\varepsilon]$. 
Redefining $T=\varepsilon>0$, we obtain the claim.
\end{proof}

\subsection{Local estimates}\label{sec:localestimates}

Let $\Omega\subset\Hyp$ be a smooth, convex, bounded domain. 
Let $u\in\Ct^{2,\alpha;1,\frac{\alpha}{2}}(\overline{\Omega}\times[0,T])$ be a solution to the nonlinear problem~\eqref{eqn:pde} as determined in \mbox{Lemma \ref{lem:shorttimeexistence}}.  
Restricting the hyperbolic background metric $g_{\Hyp}$ to $\Omega$, we obtain the Yamabe flow $g(t)=u(\cdot,t)g_{\Hyp}$ on $\Omega$ with initial metric $g_0=u_0g_{\Hyp}$.
In order to estimate the scalar curvature $\Rsc=\Rsc_{g(t)}$ of $(\Omega,g(t))$ by means of the maximum principle, we will assume $u_0\in\Ct^{4,\alpha}(\overline{\Omega})$ such that $\Rsc_{g_0}\in\Ct^{2,\alpha}(\overline{\Omega})$ and specify the parabolic boundary data $\phi$ explicitly. 
We define the function $\rel\in\Ct^{2,\alpha}(\overline{\Omega})$ by
\begin{align*}
\rel(x)&\vcentcolon=-u_0(x)\,\Rsc_{g_0}(x)-m(m-1),
\end{align*}
which is the relative initial velocity of the Yamabe flow in question compared to the ``big bang''-Yamabe flow $m(m-1)tg_{\Hyp}$. 
Defining the constant 
\begin{align}\label{eqn:constantkappa}
\kappa&\vcentcolon=\max_{x\in\overline{\Omega}}
\Bigl\{\abs{\Rsc_{g_{0}}(x)},\frac{m(m-1)}{u_0(x)}\Bigr\}.
\end{align}
we have $\abs{\rel}\leq 2u_0\kappa$. 
For $s\geq0$, let
\begin{align*}
\psi(s)&\vcentcolon=\tfrac{1}{3}+\tfrac{1}{3}(s-1)^3\chi_{[0,1]}(s),
\end{align*}
where $\chi_{[0,1]}$ denotes the characteristic function of the interval $[0,1]$. 
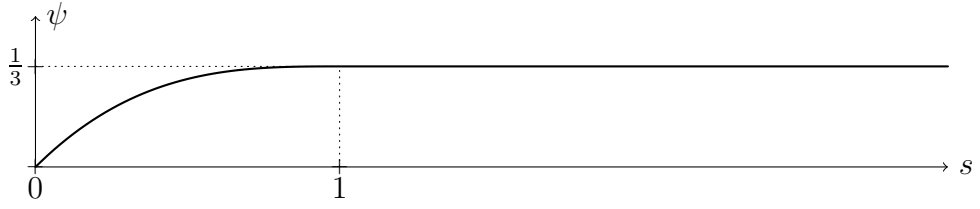
\begin{figure}[htp]
\centering
\begin{tikzpicture}[line join=round,line cap=round,scale=4]
\draw[thick,domain=0:1, smooth, variable=\x]
plot ({\x},{(1+(\x-1)^3)/3})
--(3,1/3);
\draw(0,0)node{$\scriptstyle+$}node[below]{$0$};
\draw[dotted](1,0)node{$\scriptstyle+$}node[below]{$1$}
|-(0,1/3)node{$\scriptstyle+$}node[left]{$\frac{1}{3}$};
\draw[->](0,0)--(3,0)node[right]{$s$};
\draw[->](0,0)--(0,1/2)node[right]{$\psi$};
\end{tikzpicture}
\caption{Graph of the function $\psi$.}
\end{figure}

As parabolic boundary data for problem \eqref{eqn:pde} we choose 
\begin{align}\label{eqn:boundarydata}
\phi(x,t)&=u_0(x)+m(m-1)t + \rel(x)\frac{\psi(\kappa t)}{\kappa}
\end{align}
which satisfies the desired inequalities 
\begin{align}\label{eqn:estonphi}
\tfrac{1}{3}u_0+m(m-1)t\leq \phi(\cdot,t)\leq\tfrac{5}{3}u_0+m(m-1)t
\end{align}
and the first order compatibility conditions \eqref{eqn:compatibility} by construction, i.e. $\phi(\cdot,0)=u_0$ and
\begin{align*}
\frac{\partial\phi}{\partial t}(x,t)&=m(m-1) + \rel(x)\psi'(\kappa t), 
&
\frac{\partial\phi}{\partial t}(x,0)&=-u_0(x)\,\Rsc_{g_0}(x).
\end{align*}
Moreover, we have $\phi\in\Ct^{2,\alpha;1,\frac{\alpha}{2}}(\overline{\Omega}\times[0,T])$ because $u_0\in\Ct^{4,\alpha}(\overline{\Omega})$ and $\Rsc_{g_0}\in\Ct^{2,\alpha}(\overline{\Omega})$ and since the derivatives 
\begin{align*}
\psi'(s)&=(s-1)^2\chi_{[0,1]}(s),
\\
\psi''(s)&=2(s-1)\chi_{[0,1]}(s)
\end{align*}
are continuous at $s=1$ and bounded in $\Interval{0,\infty}$. 
We observe that for any $s\in\Interval{0,\infty}$
\begin{align}
\abs[\big]{\psi(s)-s\,\psi'(s)}
&=\abs[\big]{\tfrac{1}{3}+\bigl(\tfrac{1}{3}(s-1)^3-s(s-1)^2\bigr)\chi_{[0,1]}(s)}
\notag\\\label{eqn:psionethird}
&=\abs[\big]{\tfrac{1}{3}-\tfrac{1}{3}(s-1)^2(1+2s)\chi_{[0,1]}(s)}
\leq\frac{1}{3}.
\end{align}
Given $\varepsilon>0$ to be chosen, the estimates \eqref{eqn:psionethird}, $\abs{\psi'}\leq1$ and $\abs{\rel}\leq 2u_0\kappa$ imply  
\begin{align}\notag
&\phi(\cdot,t)-(t+\varepsilon)\,\frac{\partial\phi}{\partial t}(\cdot,t)
\\\notag
&=u_0+\rel\frac{\psi(\kappa t)}{\kappa}
-\varepsilon m(m-1)
-(t+\varepsilon)\rel\psi'(\kappa t)
\\\notag
&=u_0-\varepsilon m(m-1)
+\frac{\rel}{\kappa}\bigl(\psi(\kappa t)-\kappa t\,\psi'(\kappa t)\bigr)
-\varepsilon\rel\psi'(\kappa t)
\\\notag
&\geq u_0-\varepsilon m(m-1)-\frac{\abs{\rel}}{3\kappa}-\varepsilon\abs{\rel} 
\\\notag
&\geq\frac{u_0}{3}-\varepsilon\bigl(m(m-1)+2u_0\kappa\bigr)
\geq0  
%\label{eqn:phiest1}
\end{align}
if $\varepsilon>0$ is chosen sufficiently small depending only on $u_0$, $\Omega$ and $m$. 
Hence, 
\begin{align}
\label{boundarycurvature1}
-\frac{1}{\phi}\frac{\partial\phi}{\partial t}&\geq-\frac{1}{t+\varepsilon}.
\end{align}
Let $0<K_0\leq\kappa$ be a constant such that $\Rsc_{g_0}\leq K_0$ in $\overline{\Omega}$. For $0\leq t<\tfrac{1}{K_0}$ we have 
\begin{align}\notag
&K_0\phi(\cdot,t)+(1-K_0t)\frac{\partial\phi}{\partial t}(\cdot,t)
\\\notag
&=K_0 u_0+\frac{K_0}{\kappa}\rel\psi(\kappa t)
+m(m-1)+(1-K_0 t)\rel\psi'(\kappa t)
\\
&=K_0 u_0 +m(m-1)+\Bigl(\frac{K_0}{\kappa}\psi(\kappa t)+(1-K_0 t)\psi'(\kappa t)\Bigr)\rel.
\label{eqn:phiest0}
\end{align}
To estimate \eqref{eqn:phiest0} we set $a\vcentcolon=\frac{K_0}{\kappa}\in[0,1]$ and $s=\kappa t$ and observe that the expression
\begin{align*}
\Xi&=a\psi(s)+(1-as)\psi'(s)
\\
&=\tfrac{a}{3}+\bigl(\tfrac{a}{3}(s-1)^3+(1-as)(s-1)^2\bigr)\chi_{[0,1](s)}
\\
&=\tfrac{a}{3}+\bigl(\tfrac{2a}{3}(1-s)^3+(1-a)(1-s)^2\bigr)\chi_{[0,1](s)}
\end{align*}
is decreasing in $s\in[0,1]$ as long as $a\leq1$ and therefore bounded from above by $a\psi(0)+\psi'(0)=1$ and from below by $\frac{a}{3}\geq0$. 
Substituting the term $0\leq\Xi\leq 1$ in \eqref{eqn:phiest0}, we conclude
\begin{align*} 
K_0\phi(\cdot,t)+(1-K_0t)\frac{\partial\phi}{\partial t}(\cdot,t)
&=K_0 u_0 +m(m-1)+\Xi\rel
\\
&=(K_0-\Xi\,\Rsc_{g_0})u_0+m(m-1)(1-\Xi)\geq0.
\end{align*}
For every $t\in\Interval{0,\tfrac{1}{K_0}}$ we obtain  
\begin{align}
\label{boundarycurvature2}
-\frac{1}{\phi}\frac{\partial\phi}{\partial t}&\leq\frac{K_0}{1-K_0 t}.
\end{align}
Finally, for all $t\geq\frac{1}{\kappa}$ we have $\psi'(\kappa t)=0$ and thus $\frac{\partial}{\partial t}\phi(\cdot,t)=m(m-1)$. 
Some of the previous estimates are illustrated in Figure \ref{fig:boundarydata}. 
\begin{figure}[htbp]
\pgfmathsetmacro{\globalscale}{7}
\pgfmathsetmacro{\yscl}{0.2}
\pgfmathsetmacro{\cc}{3*(3-1)}
\pgfmathsetmacro{\uo}{1.5}
\pgfmathsetmacro{\Ro}{4}
\pgfmathsetmacro{\kk}{max(\cc/\uo,\Ro,6)}
\pgfmathsetmacro{\domain}{2/3}
\begin{center}
\begin{tikzpicture}[
line cap=round,line join=round,
scale=\globalscale,yscale=\yscl,baseline={(0,0)},
declare function={
psi(\x)=(1+(\x-1)*(\x-1)*(\x-1))/3;
}]
\draw[thick,domain=0:1/\kk, samples=50 , variable=\t]
plot ({\t},{\uo+\cc*\t-(\uo*\Ro+\cc)*psi(\kk*\t)/\kk})
--++(\domain-1/\kk,\cc*\domain-\cc/\kk);
\draw[densely dashed,yscale=1/\yscl](0,0)--(\domain,\domain*\cc*\yscl)
node[sloped,very near end,below]{$\scriptstyle m(m-1)t$};
\draw[densely dotted](1/\kk,0)--++(0,{\uo+\cc/\kk-(\uo*\Ro+\cc)/(3*\kk)});
\draw[densely dotted](0,\uo)--+($-0.1*(1,-\uo*\Ro)$);
\draw[densely dotted](0,\uo)--+($ 0.1*(1,-\uo*\Ro)$);
\draw[->](0,0)--(\domain,0)node[right]{$t$};
\draw[->](0,0)--(0,\uo+\domain*\cc)node[right]{$\phi$};
\draw(0,0)node{$\scriptstyle+$}node[below]{$0$};
\draw(1/\kk,0)node{$\scriptstyle+$}node[below]{$\frac{1}{\kappa}$};
\draw(0,\uo)node{$\scriptstyle+$}node[left]{$u_0$};
\pgfresetboundingbox
\useasboundingbox(-0.2,-0.5)rectangle(\domain+0.1,0.5+\uo+\domain*\cc);
\end{tikzpicture}
\hfill
\pgfmathsetmacro{\Ro}{-8}
\pgfmathsetmacro{\kk}{max(\cc/\uo,\Ro,6)}
\begin{tikzpicture}[
line cap=round,line join=round,
scale=\globalscale,yscale=\yscl,baseline={(0,0)},
declare function={
psi(\x)=(1+(\x-1)*(\x-1)*(\x-1))/3;
}]
\draw[thick,domain=0:1/\kk , variable=\t]
plot ({\t},{\uo+\cc*\t-(\uo*\Ro+\cc)*psi(\kk*\t)/\kk})
--++(\domain-1/\kk,\cc*\domain-\cc/\kk);
\draw[densely dashed,yscale=1/\yscl](0,0)--(\domain,\domain*\cc*\yscl)
node[sloped,very near end,below]{$\scriptstyle m(m-1)t$};
\draw[densely dotted](1/\kk,0)--++(0,{\uo+\cc/\kk-(\uo*\Ro+\cc)/(3*\kk)});
\draw[densely dotted](0,\uo)--+($-0.1*(1,-\uo*\Ro)$);
\draw[densely dotted](0,\uo)--+($ 0.1*(1,-\uo*\Ro)$);
\draw[->](0,0)--(\domain,0)node[right]{$t$};
\draw[->](0,0)--(0,\uo+\domain*\cc)node[right]{$\phi$};
\draw(0,0)node{$\scriptstyle+$}node[below]{$0$};
\draw(1/\kk,0)node{$\scriptstyle+$}node[below]{$\frac{1}{\kappa}$};
\draw(0,\uo)node{$\scriptstyle+$}node[left]{$u_0$};
\pgfresetboundingbox
\useasboundingbox(-0.2,-0.5)rectangle(\domain+0.1,0.5+\uo+\domain*\cc);
\end{tikzpicture} 
\end{center}
 
\begin{center}
\begin{tikzpicture}[
line cap=round,line join=round,
scale=\globalscale,yscale={\yscl/2},
declare function={
psi(\x)=(1+(\x-1)*(\x-1)*(\x-1))/3;
dpsi(\x)=(\x-1)*(\x-1);
}]
\draw[thick,domain=0:1/\kk, samples=50 , variable=\t]
plot ({\t},{-(\cc-(\uo*\Ro+\cc)*dpsi(\kk*\t))/
(\uo+\cc*\t-(\uo*\Ro+\cc)*psi(\kk*\t)/\kk)});
\draw[thick,domain=1/\kk:\domain, samples=50 , variable=\t]
plot ({\t},{-(\cc)/(\uo+\cc*\t-(\uo*\Ro+\cc)*(1/3)/\kk)});
\draw[densely dashed, domain=1/\kk:\domain, samples=50 , variable=\t]plot ({\t},{-1/\t});
\draw[densely dotted](1/\kk,0)--++(0,{-\kk});
\draw[->](0,0)--(\domain,0)node[right]{$t$};
\draw[->](0,-\kk)--(0,\kk-1)node[right]{$\Rsc|_{\partial\Omega}=-\frac{1}{\phi}\frac{\partial\phi}{\partial t}$};
\draw(0,0)node{$\scriptstyle+$};
\draw(1/\kk,0)node{$\scriptstyle+$}node[below=1ex,fill=white,inner sep=1pt]{$\frac{1}{\kappa}$};
\draw(0,\Ro)node{$\scriptstyle+$}node[left]{$\Rsc_{g_0}$};
\draw(\domain,-1/\domain)node[below left=1ex]{$-\frac{1}{t}$};
\pgfresetboundingbox
\useasboundingbox(-0.2,-\kk)rectangle(\domain+0.1,\kk);
\end{tikzpicture}
\hfill
\pgfmathsetmacro{\Ro}{-8}
\pgfmathsetmacro{\kk}{max(\cc/\uo,\Ro,6)}
\begin{tikzpicture}[
line cap=round,line join=round,
scale=\globalscale,yscale={\yscl/2},
declare function={
psi(\x)=(1+(\x-1)*(\x-1)*(\x-1))/3;
dpsi(\x)=(\x-1)*(\x-1);
}]
\draw[thick,domain=0:1/\kk, samples=50 , variable=\t]
plot ({\t},{-(\cc-(\uo*\Ro+\cc)*dpsi(\kk*\t))/
(\uo+\cc*\t-(\uo*\Ro+\cc)*psi(\kk*\t)/\kk)});
\draw[thick,domain=1/\kk:\domain, samples=50 , variable=\t]
plot ({\t},{-(\cc)/(\uo+\cc*\t-(\uo*\Ro+\cc)*(1/3)/\kk)});
\draw[densely dashed, domain=1/(\kk+3):\domain, samples=50 , variable=\t]plot ({\t},{-1/\t});
\draw[densely dotted](1/\kk,0)--++(0,{-\kk});
\draw[->](0,0)--(\domain,0)node[right]{$t$};
\draw[->](0,-\kk-3)--(0,\kk-4)node[right]{$\Rsc|_{\partial\Omega}=-\frac{1}{\phi}\frac{\partial\phi}{\partial t}$};
\draw(1/\kk,0)node{$\scriptstyle+$}node[below=1ex,fill=white,inner sep=1pt]{$\frac{1}{\kappa}$};
\draw(0,0)node{$\scriptstyle+$};
\draw(0,\Ro)node{$\scriptstyle+$}node[left]{$\Rsc_{g_0}$};
\draw(\domain,-1/\domain)node[below left=1ex]{$-\frac{1}{t}$};
\pgfresetboundingbox
\useasboundingbox(-0.2,-\kk-3)rectangle(\domain+0.1,\kk-3);
\end{tikzpicture}
\end{center}
\caption{The evolution of $\phi$ (above) and $\Rsc$ (below) on $\partial\Omega$ for different initial values.}
\label{fig:boundarydata}
\end{figure}
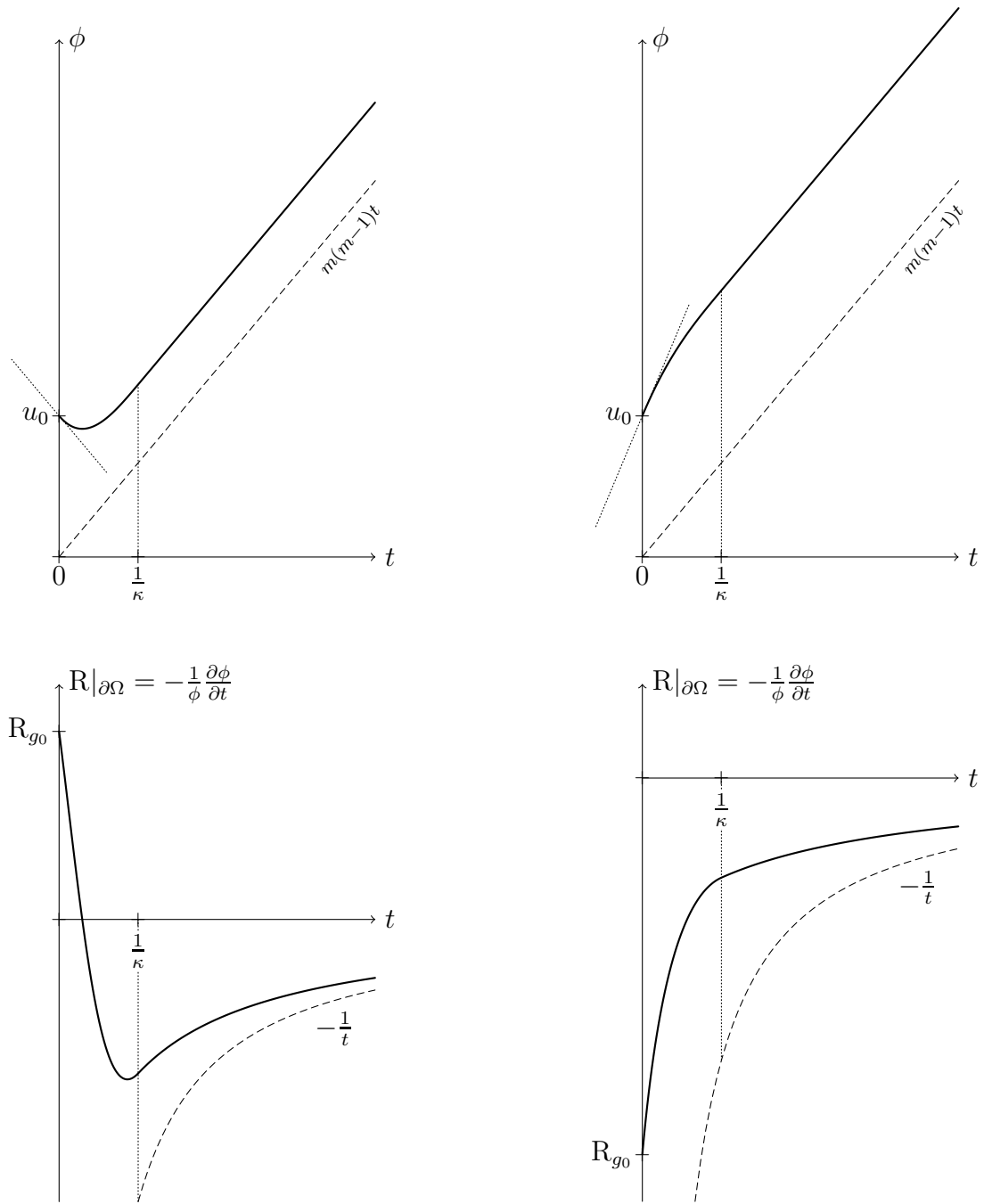

\begin{lem}[Scalar curvature bound]\label{lem:Rsc}
Let $0\leq K_0\in\R$ be a constant such that $\Rsc_{g_0}\leq K_0$ in ${\Omega}$. 
Let $T_0\vcentcolon=(\max\{K_0,\frac{1}{T}\})^{-1}$. 
Then, there exists $\varepsilon>0$ depending only on $u_0$, $\Omega$ and $m$ such that for all $(x,t)\in\Omega\times\Interval{0,T_0}$  
\begin{align*}
-\frac{1}{t+\varepsilon}\leq\Rsc(x,t)\leq \frac{K_0}{1-K_0\,t}. 
\end{align*}
\end{lem}

\begin{proof}
In $\Omega\times[0,T]$ we can express the scalar curvature in the form $\Rsc=-\frac{1}{u}\frac{\partial u}{\partial t}$. 
In particular, on $\partial\Omega\times[0,T]$ we have 
\begin{align*} 
{\Rsc|}_{\partial\Omega\times[0,T]}
=-\frac{1}{\phi}\frac{\partial\phi}{\partial t},
\end{align*}
where the right hand side satisfies the lower bound \eqref{boundarycurvature1} and the upper bound \eqref{boundarycurvature2}.
Scalar curvature evolves by the equation (see \cite[Lemma 2.2]{Chow1992})
\begin{align}
\label{eqn:evoR}
\tfrac{\partial}{\partial t}\Rsc=(m-1)\Delta_{g(t)}\Rsc+\Rsc^2
\quad\text{ in $\Omega\times[0,T]$.}
\end{align}
Let $w(t)=\frac{K_0}{1-K_0 t}$ for $t\in\Interval{0,T_0}$ with $\frac{\dd w}{\dd t}=w^2$. 
Then $\Rsc\leq w$ on $(\Omega\times\{0\})\cup(\partial\Omega\times\Interval{0,T_0})$ by \eqref{boundarycurvature2}. 
Moreover, we have 
\begin{align*}
\tfrac{\partial}{\partial t}(\Rsc-w)
-(m-1)\Delta_{g(t)}(\Rsc-w)-(\Rsc+w)(\Rsc-w)=0
\quad\text{ in $\Omega\times\Interval{0,T_0}$.}
\end{align*}
Let $T_1<T_0\leq T$ be fixed. 
Then, $(\Rsc+w)$ is bounded from above in $\Omega\times[0,T_1]$ and the inequality $\Rsc\leq w$ in $\Omega\times[0,T_1]$ follows from the parabolic maximum principle (Proposition \ref{prop:parabolicMP}). 
Since $T_1<T_0$ is arbitrary, we have $\Rsc\leq w$ in $\Omega\times\Interval{0,T_0}$.

Let $\varepsilon>0$ be sufficiently small depending only on $u_0$, $\Omega$ and $m$, such that  \eqref{boundarycurvature1} holds and such that additionally, $\Rsc_{g_0}\geq-\frac{1}{\varepsilon}$ in $\Omega$. 
In the argument above we replace $w(t)$ by $-\frac{1}{t+\varepsilon}$ and conclude $\Rsc\geq-\frac{1}{t+\varepsilon}$ analogously.  
\end{proof}

\begin{lem}[Upper and lower bound]\label{lem:pointwise}
Let $0<u\in\Ct^{2,\alpha;1,\frac{\alpha}{2}}(\overline{\Omega}\times[0,T])$ be a solution to problem \eqref{eqn:pde} with boundary data \eqref{eqn:boundarydata} and bounded initial data $u_0>0$. 
Then, for every $0\leq t\leq T$, 
\begin{align*}
m(m-1)t+\tfrac{1}{3}\min_{\overline{\Omega}}u_0\leq u(\cdot,t)\leq m(m-1)t+\tfrac{5}{3}\max_{\overline{\Omega}}u_0.
\end{align*}
\end{lem}

\begin{proof}
From the equation for $u$, we deduce that given any constant $c\in\R$ the function $w(\cdot,t)=u(\cdot,t)-m(m-1)t-c$ satisfies 
\begin{align}\label{lem:pointwise-proof2}
\frac{1}{m-1}\frac{\partial w}{\partial t}
-\frac{\Delta_{g_{\Hyp}} w}{u}-\frac{(m-6)\sk{\nabla u,\nabla w}_{g_{\Hyp}}}{4u^2}=0
\quad\text{ in $\Omega\times[0,T]$.}
\end{align} 
Since $u>0$, equation \eqref{lem:pointwise-proof2} is uniformly parabolic. 
For $c=\frac{1}{3}\min_{\overline{\Omega}}u_0$ (respectively $c=\frac{5}{3}\max_{\overline{\Omega}}u_0$ ) we have $w\geq0$ (respectively $w\leq0$) on $(\partial\Omega\times[0,T])\cup(\Omega\times\{0\})$ by \eqref{eqn:estonphi} and the parabolic maximum principle (Proposition~\ref{prop:parabolicMP}) implies $w\geq0$ (respectively $w\leq0$) in $\Omega\times[0,T]$.  
\end{proof}

\begin{lem}[Uniqueness on bounded domains]\label{lem:Uniquenessonboundeddomains}
Let $u,v\in\Ct^{2,\alpha;1,\frac{\alpha}{2}}(\overline{\Omega}\times[0,T])$ be two positive solutions of problem~\eqref{eqn:pde} with equal initial and boundary data. 
Then $u=v$. 
\end{lem}

\begin{proof}
With derivatives and inner products taken with respect to $g_{\Hyp}$, we have by \eqref{eqn:Yamabe-flow} 
\begin{align*}
\frac{1}{m-1}\frac{\partial}{\partial t}(u-v)
&=\frac{\Delta u}{u}-\frac{\Delta v}{v}+\frac{m-6}{4}\Bigl(\frac{\abs{\nabla u}^2}{u^2}-\frac{\abs{\nabla v}^2}{v^2}\Bigr)
\\[1ex]
&=\frac{\Delta (u-v)}{u}+\frac{m-6}{4u^2}\sk[\big]{\nabla(u+v),\nabla(u-v)} 
\\&\hphantom{{}={}}-\frac{\Delta v}{uv}(u-v) -\frac{m-6}{4}\frac{\abs{\nabla v}^2}{u^2v^2}(u+v)(u-v)
\end{align*}
which can be considered as linear parabolic equation for $u-v$ with bounded coefficients because $u,v\in\Ct^{2,\alpha;1,\frac{\alpha}{2}}(\overline{\Omega}\times[0,T])$ are uniformly bounded away from zero and from above and $\abs{\nabla u}$, $\abs{\nabla v}$, $\Delta v$ are bounded functions in $\overline{\Omega}$.   
Since $(u-v)$ vanishes along $(\Omega\times\{0\})\cup(\overline{\Omega}\times[0,T])$, the parabolic maximum principle (Proposition \ref{prop:parabolicMP}) implies $u-v=0$ in $\overline{\Omega}\times[0,T]$ as claimed. 
\end{proof}

\begin{lem}[Local H\"older estimate]
\label{lem:hoelderestimate}
Let $u\in\Ct^{2,\alpha;1,\frac{\alpha}{2}}(\overline{\Omega}\times[0,T])$ be a solution to problem \eqref{eqn:pde} with boundary data \eqref{eqn:boundarydata} and $T<\frac{1}{K_0}$, where $K_0\geq0$ is an upper bound for $\Rsc_{g_0}$ in $\Omega$. 
Then, there exists a constant $C$, depending only on the dimension $m$, the initial data $u_0$, the constant $K_0$ and the domain $\Omega$ such that for every $t\in[0,T]$
\begin{align*}
\nm[]{u(\cdot,t)}_{\Ct^{2,\alpha}(\overline{\Omega})}
&\leq C.
\end{align*}
\end{lem}

\begin{proof}
Let $U=u^{\eta}$ be the corresponding solution to equation \eqref{eqn:Yamabe-flow-eta}, i.e.  
\begin{align}\label{eqn:Yamabe-flow-eta2}
U^{\frac{1}{\eta}}\tfrac{\partial}{\partial t}U
=-\eta U^{1+\frac{1}{\eta}}\,\Rsc&=(m-1)\bigl(m\eta U+\Delta_{g_{\Hyp}} U\bigr). 
\end{align}
Lemmata \ref{lem:Rsc} and \ref{lem:pointwise} yield uniform bounds on the function $U$ and the scalar curvature $\Rsc$ in $\Omega\times[0,T]$.  
Therefore, equation \eqref{eqn:Yamabe-flow-eta2} implies
\begin{align*}
\nm[\big]{-\Delta_{g_{\Hyp}} U(\cdot,t)}_{L^{\infty}(\overline{\Omega})}
&\leq C, &
\nm[\big]{\tfrac{\partial}{\partial t}U(\cdot,t)}_{L^{\infty}(\overline{\Omega})}
&\leq C
\end{align*}
for every $t\in[0,T]$, where $C$ is a finite constant depending only on $m$, $u_0$, $K_0$ and $\Omega$.  
Elliptic $L^p$-Theory \cite[\textsection\,9.5]{Gilbarg2001} implies $\nm{U(\cdot,t)}_{W^{2,p}(\overline{\Omega})}\leq C$ for every $1<p<\infty$. 
With $p>m$, Sobolev's embedding theorem implies $\nm{U(\cdot,t)}_{\Ct^{1,\alpha}(\overline{\Omega})}\leq C$. 
In particular, since $U=u^{\eta}$ is bounded away from zero by Lemma \ref{lem:pointwise}, we have 
\begin{align*}
\nm[\big]{U^{-\frac{1}{\eta}}}_{\Ct^{0,\alpha;0,\frac{\alpha}{2}}(\overline{\Omega}\times[0,T])}\leq C. 
\end{align*}
Hence, the equation 
\begin{align}\label{eqn:linearhilfs}
\frac{1}{m-1}\frac{\partial V}{\partial t}
&=\bigl(m\eta V+\Delta_{g_{\Hyp}} V\bigr)U^{-\frac{1}{\eta}}
\quad\text{ in $\Omega\times[0,T]$}
\end{align}
has sufficiently regular coefficients for linear parabolic theory \cite[\textsection\,IV.5, Theorem 5.2]{Ladyzenskaja1967} to apply: 
It follows that  $V=U$ is the unique solution to \eqref{eqn:linearhilfs} with the given initial and boundary data. 
Moreover, $U$ satisfies
\begin{align*}
\nm{U}_{\Ct^{2,\alpha;1,\frac{\alpha}{2}}(\overline{\Omega}\times[0,T])}\leq C.
\end{align*}
Since $U$ is bounded away from zero in $\overline{\Omega}\times[0,T]$, the claim follows. 
\end{proof}

\begin{cor}[Extension in time]\label{cor:extentionintime}
If the initial scalar curvature $\Rsc_{g_0}$ is bounded from above by $K_0\geq0$ in $\Hyp$, then problem \eqref{eqn:pde} with boundary data \eqref{eqn:boundarydata} is solvable in $\Omega\times\Interval{0,\frac{1}{K_0}}$ for every smooth, bounded domain $\Omega\subset\Hyp$.
\end{cor}

\begin{proof}
According to Lemma \ref{lem:shorttimeexistence}, problem \eqref{eqn:pde} is solvable in $\Omega\times[0,T]$ for some $T>0$. 
Let $T_{*}>0$ be the maximal existence time, i.\,e. the supremum over all $T>0$ such that problem \eqref{eqn:pde} has a solution defined in $\Omega\times[0,T]$. 
By Lemma \ref{lem:Uniquenessonboundeddomains}, two such solutions agree on their common domain, therefore there exists a solution $u$ defined on $\overline{\Omega}\times\Interval{0,T_{*}}$. 
Suppose, that for some $\Omega$ the maximal existence time is $T_{*}<\frac{1}{K_0}$. 
Then, Lemma \ref{lem:hoelderestimate} implies that $u$ can be extended to $u\in\Ct^{2,\alpha;1,\frac{\alpha}{2}}(\overline{\Omega}\times[0,T_{*}])$ and that $u(\cdot,T_{*})\in\Ct^{2,\alpha}(\overline{\Omega})$ is suitable initial data for problem \eqref{eqn:pde}. 
The boundary data \eqref{eqn:boundarydata} are defined also for $t\geq T_{*}$ and they are compatible with $u(\cdot,T_{*})$ at time $T_{*}$. 
Therefore, we may apply Lemma \ref{lem:shorttimeexistence} to extend the solution regularly in time in contradiction to the maximality of $T_{*}$. 
\end{proof}

\subsection{Uniform estimates}

We assume that the initial metric $g_0=u_0g_{\Hyp}$ and its scalar curvature satisfy the upper bounds $u_0\leq C_0$ and $\Rsc_{g_0}\leq K_0$ in $\Hyp$ with some constant $K_0\geq 0$. 
Let $0<T<\frac{1}{K_0}$ be fixed. 
From the previous section we recall that for any smooth, bounded domain $\Omega\subset\Hyp$, there exists a uniformly bounded solution $u$ of \eqref{eqn:pde} on $\Omega\times[0,T]$. 
However, the previous H\"older estimates on $u$ may depend on the domain $\Omega$. 
In the following, we derive independent bounds. 
As before, spatial derivatives and inner products are taken with respect to the hyperbolic background metric but in the following we will suppress the index $g_{\Hyp}$ to ease notation. 

\begin{lem}[Uniform gradient estimate]\label{lem:uniformestimate}
Let $\Bl_\ell$ be the metric ball of radius $\ell>1$ around the origin in $(\Hyp,g_{\Hyp})$. 
Let $u\in\Ct^{2,\alpha;1,\frac{\alpha}{2}}(\Bl_\ell\times[0,T])$ be a solution to problem  
\eqref{eqn:pde} with boundary data \eqref{eqn:boundarydata} as constructed in Corollary \ref{cor:extentionintime}. 
Then, $U=u^{\eta}$ with $\eta=\frac{m-2}{4}$ satisfies 
\begin{align*}
\abs{\nabla U}^2\leq C \quad \text{ in $\Bl_{\ell-1}\times[0,T]$, }
\end{align*}
where the constant $C$ depends on the dimension $m$ and the constants $C_0$, $K_0$, $T$ but not on $\ell$. 
Similar bounds hold for higher derivatives of $U$. 
\end{lem}

\begin{proof}
Let $p\in\mathbb{R}$ be an exponent. 
As in \cite{Ma1999}, we consider the function $w=U^p\abs{\nabla U}^2$ and compute
\begin{align*}
\frac{1}{m-1}\frac{\partial w}{\partial t}
&=p U^{p-1-\frac{1}{\eta}}\frac{U^{\frac{1}{\eta}}}{m-1}\frac{\partial U}{\partial t}\abs{\nabla U}^2
+2\sk[\Big]{\nabla U,\frac{U^{p}}{m-1}\nabla\frac{\partial U}{\partial t}}.
\end{align*}
We recall $\Rsc=-\frac{1}{\eta U}\frac{\partial U}{\partial t}$ and $\frac{1}{m-1}U^{\frac{1}{\eta}}\frac{\partial U}{\partial t}=m\eta U+\Delta U$ from equation \eqref{eqn:Yamabe-flow-eta}. 
Since 
\begin{align*}
\frac{U^{p}}{m-1}\nabla\frac{\partial U}{\partial t}
&=U^{p-\frac{1}{\eta}}\nabla\Bigl(\frac{U^{\frac{1}{\eta}}}{m-1}\frac{\partial U}{\partial t}\Bigr)
-\frac{1}{\eta}\frac{U^{p-1}}{m-1}\frac{\partial U}{\partial t}\nabla U
\\[.5ex]
&=U^{p-\frac{1}{\eta}}\bigl(m\eta\nabla U+\nabla\Delta U\bigr)
+\frac{\Rsc}{m-1}U^p\nabla U,
\end{align*}
we have
\begin{align}
\frac{1}{m-1}\frac{\partial w}{\partial t} 
&=p U^{p-1-\frac{1}{\eta}}\bigl(m\eta U+ \Delta U\bigr)\abs{\nabla U}^2
+2m\eta U^{p-\frac{1}{\eta}}\abs{\nabla U}^2
\notag\\&\hphantom{{}={}}
+2U^{p-\frac{1}{\eta}}\sk{\nabla U,\nabla\Delta U}
+\frac{2\Rsc}{m-1}U^p\abs{\nabla U}^2
\notag\\[1ex]
&=(p+2)m\eta U^{p-\frac{1}{\eta}}\abs{\nabla U}^2 
+\frac{2w\Rsc}{m-1}
\label{dwdt}
\\&\hphantom{{}={}}\notag
+\bigl(p U^{p-1}\abs{\nabla U}^2\Delta U + 2U^p\sk{\nabla U,\nabla\Delta U}\bigr)U^{-\frac{1}{\eta}}.
\end{align}
Bochner's identity implies 
\begin{align}\label{Ricci-Identity}
\Delta\tfrac{1}{2}\abs{\nabla U}^2
&=\abs{\nabla\nabla U}^2+\sk{\nabla U,\nabla\Delta U}+\Ric(\nabla U,\nabla U).
\end{align}
Together with $\Ric_{g_{\Hyp}}=-(m-1)g_{\Hyp}$, we apply \eqref{Ricci-Identity} in the following computation. 
\begin{align*}
\nabla w
&=p U^{p-1}\abs{\nabla U}^2\nabla U+U^p\nabla\abs{\nabla U}^2,
\\[1ex]
\Delta w
&=p(p-1)U^{p-2}\abs{\nabla U}^4
+2p U^{p-1}\sk{\nabla U,\nabla\abs{\nabla U}^2}
\\&\hphantom{{}={}}
+p U^{p-1}\abs{\nabla U}^2\Delta U
+U^p\Delta\abs{\nabla U}^2
\\[1ex]
&=p(p-1)U^{p-2}\abs{\nabla U}^4
+2p U^{-1}\sk{\nabla U,\nabla w}
-2p^2 U^{p-2}\abs{\nabla U}^4
\\&\hphantom{{}={}}
+p U^{p-1}\abs{\nabla U}^2\Delta U
+2U^p\bigl(\abs{\nabla\nabla U}^2+\sk{\nabla U,\nabla\Delta U}-(m-1)\abs{\nabla U}^2\bigr).
\end{align*}
Hence,
\begin{align}
&\bigl(p U^{p-1}\abs{\nabla U}^2\Delta U
+2U^p\sk{\nabla U,\nabla\Delta U}\bigr)
\notag\\
&=\Delta w-2p U^{-1}\sk{\nabla U,\nabla w}
\notag\\&\hphantom{{}={}}\label{laplaceterms}
+p(p+1)U^{p-2}\abs{\nabla U}^4+2(m-1) U^{p}\abs{\nabla U}^2-2U^{p}\abs{\nabla\nabla U}^2. 
\end{align}
We insert \eqref{laplaceterms} into \eqref{dwdt} and resubstitute $w=U^p\abs{\nabla U}^2$ to obtain
\begin{align*}
\frac{1}{m-1}\frac{\partial w}{\partial t} 
&=\bigl((p+2)m\eta+2(m-1)\bigr)U^{-\frac{1}{\eta}} w
+\frac{2w\Rsc}{m-1}
\\&\hphantom{{}={}}
+\bigl(\Delta w-2p U^{-1}\sk{\nabla U,\nabla w}\bigr)U^{-\frac{1}{\eta}}
\\&\hphantom{{}={}}
+p(p+1)U^{-p-2-\frac{1}{\eta}}w^2-2U^{p-\frac{1}{\eta}}\abs{\nabla\nabla U}^2. 
\end{align*}
Choosing $p=-\frac{1}{2}$ and deducing $\frac{2}{m-1}\Rsc\leq K_1$ from Lemma \ref{lem:Rsc} for some constant $K_1\geq 0$ depending on $K_0$ and $T$, we obtain
\begin{align}
\frac{1}{m-1}\frac{\partial w}{\partial t} 
&\leq\bigl(\tfrac{3}{2}m\eta+2m-2\bigr)U^{-\frac{1}{\eta}} w
+K_1 w
\notag\\&\hphantom{{}={}}\label{est:wohnechi}
+\bigl(\Delta w+ U^{-1}\sk{\nabla U,\nabla w}\bigr)U^{-\frac{1}{\eta}}
-\tfrac{1}{4}U^{-\frac{3}{2}-\frac{1}{\eta}}w^2.
\end{align}
Let $\varphi\colon\R\to[0,1]$ be a smooth, non-increasing cutoff function satisfying $\varphi(x)=1$ for $x\leq\ell-1$ and $\varphi(x)=0$ for $x\geq\ell$. 
Let $r\colon\Hyp\to\Interval{0,\infty}$ be the Riemannian distance function from the origin in $(\Hyp,g_{\Hyp})$ and let $\chi=\varphi\circ r$. 
Before we multiply both sides of equation \eqref{est:wohnechi} by $\chi$, we compute 
\begin{align}
\chi\Delta w&=\Delta(\chi w)-w\Delta\chi-2\sk{\nabla w,\nabla\chi}
\notag\\\label{eqn:chiw1}
&=\Delta(\chi w)-\frac{2}{\chi}\sk[\big]{\nabla(w\chi),\nabla\chi}
+\Bigl(\frac{2\abs{\nabla\chi}^2}{\chi}-\Delta\chi\Bigr)w. 
\end{align}
Lemma \ref{lem:Laplacehypdist} stating $\Delta r\leq2(m-1)$ on $\Hyp\setminus\Bl_1$ and Lemma \ref{lem:cutoffest} about cutoff functions (both given in the appendix) provide an estimate of the last term: 
There exists a constant $c_m$ depending only on the dimension $m$ such that 
\begin{align}\label{eqn:chiw2}
\Bigl(\frac{2\abs{\nabla\chi}^2}{\chi}-\Delta\chi\Bigr)
=\frac{2\abs{\varphi'(r)}^2}{\varphi(r)}-\varphi''(r)-\varphi'(r)\Delta r
&\leq c_m\sqrt{\chi}
\end{align}
and such that $\chi^{-3}\abs{\nabla\chi}^4\leq c_m^2$ which will be used later.  
\eqref{eqn:chiw1} and \eqref{eqn:chiw2} lead to
\begin{align}
\frac{1}{m-1}\frac{\partial(\chi w)}{\partial t} 
&\leq\Bigl(\bigl(\tfrac{3}{2}m\eta+2m\bigr)U^{-\frac{1}{\eta}}
+K_1\Bigr)(\chi w)
\notag\\&\hphantom{{}={}}\notag
+U^{-\frac{1}{\eta}}\chi\Delta w
+\sk[\big]{\nabla U,\chi\nabla w} U^{-1-\frac{1}{\eta}}
-\tfrac{1}{4}U^{-\frac{3}{2}-\frac{1}{\eta}}\chi w^2 
\\[1ex]\notag
&\leq\Bigl(\bigl(\tfrac{3}{2}m\eta+2m\bigr)U^{-\frac{1}{\eta}}
+K_1\Bigr)(\chi w)
\notag\\&\hphantom{{}={}}\label{eqn:chiw3}
+\Bigl(\Delta(\chi w)-\frac{2}{\chi}\sk[\big]{\nabla(w\chi),\nabla\chi}\Bigr)U^{-\frac{1}{\eta}}
+c_m\sqrt{\chi}\,w\,U^{-\frac{1}{\eta}}
\\&\hphantom{{}={}}\notag
+\sk[\big]{\nabla U,\nabla(\chi w)}U^{-1-\frac{1}{\eta}}
-\sk[\big]{\nabla U,w\nabla\chi}U^{-1-\frac{1}{\eta}}
-\tfrac{1}{4}U^{-\frac{3}{2}-\frac{1}{\eta}}\chi w^2. 
\end{align}
Young's inequality for $a,b\in\mathbb{R}$, $\delta>0$ and $p,q>1$ with $\frac{1}{p}+\frac{1}{q}=1$ states that 
\begin{align*}
\abs{ab}\leq \delta\abs{a}^p+\tfrac{1}{q}(p\delta)^{1-q}\abs{b}^q.
\end{align*}
We apply it with $p=\frac{4}{3}$ and $q=4$ and recall $w=U^{-\frac{1}{2}}\abs{\nabla U}^2$, i.e. $\abs{\nabla U}=w^{\frac{1}{2}}U^{\frac{1}{4}}$ to obtain  
\begin{align*}
-\sk[\big]{\nabla U,w\nabla\chi}U^{-1-\frac{1}{\eta}}
&\leq w\abs{\nabla U}U^{-1-\frac{1}{\eta}}\,\abs{\nabla \chi}
\\[1ex]
&=\bigl(U^{-\frac{9}{8}}\chi^{\frac{3}{4}}w^{\frac{3}{2}}\bigr)
\bigl(U^{\frac{3}{8}}\chi^{-\frac{3}{4}}\abs{\nabla\chi}\bigr)
U^{-\frac{1}{\eta}} 
\\[1ex]
&\leq\delta U^{-\frac{3}{2}-\frac{1}{\eta}}\chi w^2 
+\frac{3^3}{4^4}\delta^{-3}\frac{\abs{\nabla\chi}^4}{\chi^3}U^{\frac{3}{2}-\frac{1}{\eta}}.
\end{align*}
Young's inequality with $p=2=q$ also yields
\begin{align*}
c_m\sqrt{\chi}\,w\,U^{-\frac{1}{\eta}}
&\leq\delta U^{-\frac{3}{2}-\frac{1}{\eta}}\chi w^2
+\frac{c_m^2}{4\delta}U^{\frac{3}{2}-\frac{1}{\eta}}.
\end{align*}
Let the sum of all terms in \eqref{eqn:chiw3} containing $\Delta(\chi w)$ or $\nabla(\chi w)$ be denoted by 
\begin{align*}
\Lambda\vcentcolon=\Bigl(\Delta(\chi w)-\frac{2}{\chi}\sk[\big]{\nabla(w\chi),\nabla\chi}\Bigr)U^{-\frac{1}{\eta}}
+\sk[\big]{\nabla U,\nabla(\chi w)}U^{-1-\frac{1}{\eta}}
\end{align*}
and let the largest of the occurring factors which depend only on the dimension $m$ be denoted by $C_m$. 
Then
\begin{align*}
\frac{1}{m-1}\frac{\partial(\chi w)}{\partial t}
&\leq C_m\bigl(U^{-\frac{1}{\eta}}+K_1\bigr)(\chi w)
+C_m\bigl(\delta^{-1}+\delta^{-3}\bigr)U^{\frac{3}{2}-\frac{1}{\eta}} 
\\&\hphantom{{}={}}\notag
-(\tfrac{1}{4}-2\delta)U^{-\frac{3}{2}-\frac{1}{\eta}}\chi w^2+\Lambda.
\end{align*}
Choosing $\delta=\frac{1}{16}$, we have $-(\tfrac{1}{4}-2\delta)=-\frac{1}{8}<0$. 
Since $-\chi w^2\leq-(\chi w)^2$ and since 
\begin{align*}
c_1(\chi w)-c_2(\chi w)^2\leq\frac{c_1^2}{4c_2},
\end{align*}
for any $c_1,c_2>0$, we obtain
\begin{align}\label{est:wmax}
\frac{\partial(\chi w)}{\partial t}-\Lambda
&\leq\tilde{C}_m\bigl(
U^{\frac{3}{2}-\frac{1}{\eta}}
+K_1^2 U^{\frac{3}{2}+\frac{1}{\eta}}
\bigr)
\end{align}
with a different constant $\tilde{C}_m$. 
By Lemma \ref{lem:pointwise}, we have 
\begin{align*}
\bigl(m(m-1)t\bigr)^{\eta}\leq U(\cdot,t)\leq C(m,C_0,T)
\end{align*}
in $\Bl_\ell\times[0,T]$. 
Since $(\frac{3}{2}-\frac{1}{\eta})\eta>-1$, the right hand side of \eqref{est:wmax} is bounded from above by a spatially constant, positive function $f\in L^1([0,T])$.
Let $F'(t)=f(t)$ with $F(0)=\max(\chi w)(0)$ define a primitive function $F$ for $t\mapsto f(t)$. 
Then, 
\begin{align*}
\frac{\partial(\chi w-F)}{\partial t}\leq
\Lambda=
U^{-\frac{1}{\eta}}\Delta(\chi w-F)
+U^{-\frac{1}{\eta}}\sk[\Big]{
\nabla(w\chi-F),\frac{\nabla U}{U}-\frac{2\nabla\chi}{\chi}}.
\end{align*}
Since $\chi w-F=-F\leq0$ on $\partial\Bl_\ell\times[0,T]$ and $\chi w\leq F$ on $\Bl_\ell\times\{0\}$, the parabolic maximum principle (Proposition \ref{prop:parabolicMP}) implies $\chi w\leq F$ in $\Bl_\ell\times[0,T]$. 
The map $t\mapsto F(t)$ is increasing and $F(T)$ depends only on $m,C_0,K_0,T$. 
Therefore, in $\Bl_{\ell-1}\times[0,T]$, we finally have
\begin{align*}
\abs{\nabla U}^2=(\chi w)\sqrt{U}\leq C(m,C_0,K_0,T). 
\end{align*}
Since the Yamabe flow equation is only of second order, similar estimates on higher derivatives of $U$ follow analogously. 
\end{proof}

\begin{proof}[Proof of Theorem \ref{thm:existence}]
Let $\eta=\frac{m-2}{4}$ and let the initial metric $g_0=u_0g_{\Hyp}$ be given by $0<U_0=u_0^{\eta}\in\Ct^{2,\alpha}(\Hyp)$. 
Let $\Bl_r$ be the metric ball of radius $r>0$ around the origin in $(\Hyp,g_{\Hyp})$. 
Then, $\Bl_1\subset\Bl_2\subset\ldots\subset\Hyp$ is an exhaustion of $\Hyp$ with smooth, bounded domains. 
We fix $0<T<\frac{1}{K_0}$ and choose $\phi$ as in \eqref{eqn:boundarydata}. 
By Corollary \ref{cor:extentionintime}, the problem
\begin{align*}\left\{
\begin{aligned}
\frac{1}{m-1}\frac{\partial U_k}{\partial t}&=\bigl(m\eta U_k+\Delta_{g_{\Hyp}} U_k\bigr)U_k^{-\frac{1}{\eta}}
&&\text{ in $\Bl_k\times[0,T]$, } \\
U_k&=\phi^{\eta}
&&\text{ on $\partial \Bl_k\times[0,T]$, } \\[.5ex] 
U_k&=U_0 
&&\text{ on $\Bl_k\times\{0\}$, }  
\end{aligned}\right.
\end{align*}
is solvable for every $k>0$. 
According to Lemma \ref{lem:uniformestimate}, the sequence $\{{U_k}|_{\Bl_1\times[0,T]}\}_{2\leq k\in\N}$ is uniformly bounded in $\Ct^{2,\alpha;1,\frac{\alpha}{2}}(\Bl_1\times[0,T])$. 
Since $\Bl_1$ is a bounded domain, the embedding $\Ct^{2,\alpha;1,\frac{\alpha}{2}}(\Bl_1\times[0,T])\hookrightarrow\Ct^{2;1}(\Bl_1\times[0,T])$ is compact and we obtain a subsequence $\Lambda_1\subset\N$ such that 
\begin{align*}
\{{U_k}|_{\Bl_1\times[0,T]}\}_{2\leq k\in\Lambda_1}
\end{align*}
converges in $\Ct^{2;1}(\Bl_1\times[0,T])$ to a solution of the Yamabe flow equation \eqref{eqn:Yamabe-flow-eta} on $\Bl_1$. 
We repeat this argument to obtain a subsequence $\Lambda_2\subset\Lambda_1$ such that 
\begin{align*}
\{{U_k}|_{\Bl_2\times[0,T]}\}_{3\leq k\in\Lambda_2}
\end{align*}
converges to a solution of \eqref{eqn:Yamabe-flow-eta} on $\Bl_2$. 
Iterating this procedure leads to a diagonal subsequence of $\{U_k\}_{2\leq k}$ which converges to a limit $U\in\Ct^{2;1}(\Hyp\times[0,T])$ satisfying the Yamabe flow equation \eqref{eqn:Yamabe-flow-eta}.
Since the bounds from Lemma \ref{lem:pointwise} are preserved in the limit, we have $m(m-1)t\leq U^{\frac{1}{\eta}}$, i.e. $(g(t))_{t\in[0,T]}$ given by $g(t)=U(\cdot,t)^{\frac{1}{\eta}}g_{\Hyp}$ is an instantaneously complete Yamabe flow. 

It remains to show that the Yamabe flow constructed above can be extended in time. 
Let $T_*$ be the supremum over all $T>0$ such that there exists a Yamabe flow $g(t)=u(\cdot,t)g_{\Hyp}$ on $\Hyp$ which is defined for $t\in[0,T]$ and  satisfies  $u(\cdot,0)=u_0$ as well as 
\begin{align}
\label{eqn:180824}
\forall t\in[0,T]:\quad
m(m-1)t\leq u(\cdot,t)&\leq m(m-1)t+\tfrac{5}{3}\sup_{\Hyp}u_0.
\end{align}
We have already shown $T_{*}\geq\frac{1}{K_0}$. 
Suppose, $T_{*}<\infty$ and let $0<\varepsilon<\frac{1}{5}T_{*}$ be arbitrary. 
For $T=T_{*}-\varepsilon$, there exists $u\colon\Hyp\times[0,T]\to\R$ satisfying $u(\cdot,0)=u_0$ together with estimate \eqref{eqn:180824} and equation \eqref{eqn:Yamabe-flow-eta}  which can be written in divergence form: 
\begin{align}\notag
\frac{1}{m-1}\frac{\partial u^{\eta+1}}{\partial t}
=\frac{\eta+1}{\eta}\frac{u}{m-1}\frac{\partial u^{\eta}}{\partial t}
%\\\notag
&=m(\eta+1)u^{\eta}+\frac{\eta+1}{\eta}\operatorname{div}_{g_{\Hyp}}(\nabla u^{\eta})
\\\label{eqn:divergenceform}
&=\frac{m(\eta+1)u^{\eta+1}}{u}+\operatorname{div}_{g_{\Hyp}}\Bigl(\frac{1}{u}\nabla u^{\eta+1}\Bigr)  
\end{align}
where we recall $\eta=\frac{m-2}{4}$. 
Around an arbitrary point $p\in\Hyp$, we choose geodesic normal coordinates $x$ and given $0<r<\tfrac{1}{3}\sqrt{T}$ we consider the parabolic cylinder  
\begin{align*}
Q_r&\vcentcolon=\{(x,t)\mid T-r^2<t<T,~\abs{x}<r\}.
\end{align*} 
According to \eqref{eqn:180824} and the choices of $r$, $\varepsilon$ and $T$, we have 
\begin{align*}
m(m-1)\tfrac{4}{9}T_{*}\leq u\leq m(m-1)T_{*}+\tfrac{5}{3}\sup_{\Hyp}u_0
\qquad\text{ in $\overline{Q}_{2r}$}
\end{align*}
because $T-(2r)^2>T-\frac{4}{9}T=\frac{5}{9}(T_{*}-\varepsilon)>\frac{4}{9}T_{*}$. 
Hence, \eqref{eqn:divergenceform} can be interpreted as a linear equation with uniformly bounded coefficient $\frac{1}{u}$. 
Therefore we may apply parabolic DeGiorgi--Nash--Moser Theory \cite[Theorem III.10.1]{Ladyzenskaja1967} (see also \cite{Trudinger1968}) to equation \eqref{eqn:divergenceform} in order to obtain 
\begin{align}\label{eqn:20180823}
\nm{u^{\eta+1}}_{\Ct^{0,\alpha;0,\frac{\alpha}{2}}(\overline{Q_r})}\leq C\bigl(m,\nm{u}_{L^{\infty}(Q_{2r})},\nm{u^{-1}}_{L^{\infty}(Q_{2r})}\bigr)
\leq C'(m,T_{*},\sup u_0) 
\end{align}
for some $0<\alpha<1$ and some constants $C$, $C'$ depending only on the indicated quantities. 
In particular, the H\"older estimate \eqref{eqn:20180823} holds uniformly in $p\in \Hyp$. 
As in the proof of \mbox{Lemma \ref{lem:hoelderestimate}} we obtain 
\begin{align*}
\nm{u^{\eta}}_{\Ct^{2,\alpha;1,\frac{\alpha}{2}}(Q_r)} \leq C''(m,T_{*},\sup u_0). 
\end{align*}
Consequently, the scalar curvature $\Rsc=-(m-1)u^{-\eta-1}(\frac{1}{\eta}\Delta u^{\eta}+m u^{\eta})$ stays uniformly bounded up to time $T$ by some constant $K_1(m,T_{*},\sup u_0)$. 
With the same methods as before, we can extend the solution, first locally in bounded domains and then via exhaustion in all of $\Hyp$.  
As initial data, we choose $u_1=u(\cdot,T-\varepsilon)$. 
This allows us build compatible boundary data from a suitable extension of $u|_{\Hyp\times[T-\varepsilon,T]}$. 
It also ensures the extended solution to be regular by Lemma \ref{lem:Uniquenessonboundeddomains}. 
The extended solution is defined in $\Hyp\times[0,T-\varepsilon+T_1]$ with any given $T_1<\frac{1}{K_1}$. 
If $\varepsilon>0$ is chosen small enough such that 
\begin{align*}
T-\varepsilon+T_1=T_{*}-2\varepsilon+T_1>T_{*}
\end{align*}
we obtain a contradiction to the maximality of $T_{*}$. 
\end{proof}

%===== UNIQUENESS ======================================

\section{Uniqueness}

This section contains the proofs of Theorems \ref{thm:uniqueness} and \ref{thm:nonexistence}. 
As before, $(\Hyp,g_{\Hyp})$ denotes hyperbolic space of dimension $m\geq3$ and 
$g_{\Eucl}=h^{-2}g_{\Hyp}$ the pullback of the Euclidean metric to $\Hyp$, where $h>0$ is a smooth function provided by the Poincar\'{e} ball model. 

\subsection{Upper and lower bounds}

The Yamabe flow constructed in the previous section satisfies the upper and lower bounds given in \eqref{eqn:180824}. 
The aim of this section is to find conditions under which any Yamabe flow necessarily satisfies such bounds. 

\begin{prop}[Upper bound]\label{prop:upper}
Let $\eta=\tfrac{m-2}{4}$. 
If $g(t)=u(\cdot,t)g_{\Hyp}$ for $t\in[0,T]$ is any Yamabe flow on $\Hyp$ with initial data $g(0)\leq b\, g_{\Eucl}$ for some constant $b\in\R$, then   
\begin{align*}
\bigl(u(\cdot,t)\bigr)^{\eta}
&\leq\bigl(m(m-1)t\bigr)^{\eta}
+\bigl(b h^{-2}\bigr)^{\eta}
\end{align*}
in $\Hyp$ for every $t\in[0,T]$.  
\end{prop} 

\begin{proof}
Let $\varphi\colon\R\to[0,1]$ be a smooth, non-increasing cutoff function satisfying $\varphi(x)=1$ for $x\leq1$ and $\varphi(x)=0$ for $x\geq2$. 
Introducing geodesic normal coordinates $(r,\vartheta)\in\interval{0,\infty}\times\Sp^{m-1}$ on $(\Hyp,g_{\Hyp})$ and a parameter $A>1$, we define the rotationally symmetric function $\phi=\varphi\circ\frac{r}{A}$ on $\Hyp$. 
As before, spatial derivatives and inner products are taken with respect to the hyperbolic background metric $g_{\Hyp}$. 
With any regular function $v\colon\Hyp\to\R$, we have 
\begin{align*}
\phi\Delta v 
&=\Delta(\phi v)
-\frac{2}{\phi}\sk[\big]{\nabla\phi,\nabla(\phi v)}
+\frac{2v}{\phi}\abs{\nabla\phi}^2
-v\Delta\phi. 
\end{align*}
Setting $f\vcentcolon=b h^{-2}$ and $v\vcentcolon=(u^{\eta}-f^{\eta})$, we denote   
\begin{align*}
\Delta(\phi v)
-\frac{2}{\phi}\sk[\big]{\nabla\phi,\nabla(\phi v)}
=\vcentcolon\Lambda.
\end{align*}
Since $f g_{\Hyp}=b\,g_{\Eucl}$ is a flat metric, it is a static solution to the Yamabe flow equation $\partial_tg=-\Rsc g$ which by \eqref{eqn:Yamabe-flow-eta} implies $-\Delta f^{\eta}=m\eta f^{\eta}$.
Thus, 
\begin{align}\notag
\frac{1}{m-1}\frac{\partial}{\partial t}
(u^{\eta}-f^{\eta})\phi
&=\bigl(m\eta u^{\eta} + \Delta u^{\eta}\bigr)\frac{\phi}{u} 
\\\notag
&=m\eta\frac{(u^{\eta}-f^{\eta})}{u}\phi 
+\frac{\phi\Delta(u^{\eta}-f^{\eta})}{u}
\\\label{MP-est-1}
&=m\eta\frac{(u^{\eta}-f^{\eta})}{u}\phi
+\frac{\Lambda}{u}
+\Bigl(\frac{2\abs{\nabla\phi}^2}{\phi}-\Delta\phi\Bigr)
\frac{(u^{\eta}-f^{\eta})}{u}.
\end{align}
Since $\nabla\phi=\frac{1}{A}(\varphi'\circ\frac{r}{A})\nabla r$, we have 
\begin{align*}
\frac{2\abs{\nabla\phi}^2}{\phi}-\Delta\phi
&=\Bigl(\frac{2\varphi'^2}{A^2\varphi}
-\frac{\varphi''}{A^2}\Bigr)\circ(\tfrac{r}{A})
-\frac{\varphi'\circ\frac{r}{A}}{A}\Delta r.
\end{align*}
As $\varphi'\circ\frac{r}{A}$ is non-positive in $\Hyp$ and identically zero in the unit ball around the origin, we may apply Lemma \ref{lem:Laplacehypdist} stating $\Delta r\leq2(m-1)$ and Lemma \ref{lem:cutoffest} about cutoff functions (both given in the appendix) to estimate  
\begin{align}
\label{MP-est-2}
\Bigl(\frac{2\abs{\nabla\phi}^2}{\phi}-\Delta\phi\Bigr)
&\leq\frac{1}{A}\Bigl(\frac{2\varphi'^2}{A\varphi}
+\frac{\abs{\varphi''}}{A}+2(m-1)\abs{\varphi'}\Bigr)\circ(\tfrac{r}{A})
\leq \frac{C}{A}\phi^{1-\frac{1}{\eta}}
\end{align}
at points where $\phi\neq0$ with a constant $C$ depending only on the dimension and the choice of $\varphi$. 
Let time $t_0\in\intervaL{0,T}$ be fixed. 
The function $[(u^{\eta}-f^{\eta})\phi](\cdot,t_0)$ has a global, non-negative maximum at some point $q_{0}\in\Hyp$ which depends on $t_0$ and the parameters $A$ and $b$. 
If $[(u^{\eta}-f^{\eta})\phi](q_0,t_0)>0$, we have
\begin{align*}
\frac{u^{\eta}-f^{\eta}}{u}(q_0,t_0)
\leq\frac{u^{\eta}-f^{\eta}}{(u^{\eta}-f^{\eta})^{\frac{1}{\eta}}}(q_0,t_0)
=(u^{\eta}-f^{\eta})^{1-\frac{1}{\eta}}(q_0,t_0).
\end{align*}
Equation \eqref{MP-est-1} and estimate \eqref{MP-est-2} then yield 
\begin{align*}
&\Bigl[\frac{1}{m-1}\frac{\partial}{\partial t}
(u^{\eta}-f^{\eta})\phi\Bigr](q_{0},t_0)
\\
&\leq \Bigl[m\eta(u^{\eta}-f^{\eta})^{1-\frac{1}{\eta}}\phi
+\tfrac{C}{A}\phi^{1-\frac{1}{\eta}}
(u^{\eta}-f^{\eta})^{1-\frac{1}{\eta}}
\Bigr](q_{0},t_0)
\\
&\leq \Bigl[(m\eta+\tfrac{C}{A})\bigl((u^{\eta}-f^{\eta})\phi\bigr)^{1-\frac{1}{\eta}}
\Bigr](q_{0},t_0).
\end{align*}
Denoting $v=u^{\eta}-f^{\eta}$ as before, we obtain 
\begin{align*}
&\limsup_{\tau\searrow0}
\frac{1}{\tau}\Bigl(\max {(v\phi)}(t_0)
-\max {(v\phi)}(t_0-\tau)\Bigr)
\\ 
&\leq \limsup_{\tau\searrow0}
\frac{1}{\tau}\Bigl((v\phi)(q_{0},t_0 )-(v\phi)(q_0,t_0-\tau) \Bigr)
=\tfrac{\partial}{\partial t}(v\phi)(q_0,t_0)
\\
&\leq(m-1)(m\eta+\tfrac{C}{A})
\bigl(\max(v\phi)(t_0)\bigr)^{1-\frac{1}{\eta}}. 
\end{align*}
The assumption $g(0)\leq b\, g_{\Eucl}$ implies $\max(u^{\eta}-f^{\eta})(0)\leq0$. 
Since $t_0>0$ is arbitrary, 
we may apply Lemma \ref{lem:backdiff2} stated in the appendix to conclude  
\begin{align}\label{eqn:20190311}
\max\bigl((u^{\eta}-f^{\eta})\phi\bigr)(t)
\leq \bigl((m-1)(m+\tfrac{C}{\eta A})t\bigr)^{\eta}. 
\end{align}
Letting $A\to\infty$ such that $\phi\to1$ pointwise on $\Hyp$, we obtain 
\begin{align*}            
u^{\eta}(\cdot,t)\leq\bigl(m(m-1)t\bigr)^{\eta}+f^{\eta}
\end{align*}
in $\Hyp$ for all $t\in[0,T]$. 
\end{proof}

A similar approach as for Proposition \ref{prop:upper} leads to a proof of Theorem \ref{thm:nonexistence}. 

\begin{proof}[Proof of Theorem \ref{thm:nonexistence}]
Suppose, $g(t)=u(\cdot,t)g_{\R^m}$ is a Yamabe flow on $\R^m$ for $t\in[0,T]$ with $g(0)=g_0$. 
Then
\begin{align}\label{eqn:euclYamabeflow}
\tfrac{1}{m-1}\tfrac{\partial}{\partial t}u^{\eta}&=\tfrac{1}{u}\Delta u^{\eta},
\end{align}
where $\eta=\frac{m-2}{4}$ and $\Delta$ denotes the Euclidean Laplacian.  
In the proof of Proposition~\ref{prop:upper}, we replace the equation for $u^{\eta}$ by \eqref{eqn:euclYamabeflow} and $f$ by $f(x)=b\abs{x}^{-4}$. Then, 
$f^{\eta}(x)=b^{\eta}\abs{x}^{2-m}$ is harmonic on $\R^m\setminus\{0\}$ which implies 
\begin{align*}
\tfrac{1}{m-1}\tfrac{\partial}{\partial t}(u^{\eta}-f^{\eta})&=\tfrac{1}{u}\Delta (u^{\eta}-f^{\eta}) \quad\text{ in $(\R^m\setminus\{0\})\times[0,T]$.}
\end{align*}
With a cutoff function $\phi$ as in \eqref{MP-est-2}, we gain   
\begin{align*}
[(u^{\eta}-f^{\eta})\phi](\cdot,t) &\leq\bigl((m-1)\tfrac{C t}{\eta A} \bigr)^{\eta}
\end{align*}
from the assumption $u(\cdot,0)\leq f$ as in \eqref{eqn:20190311}. 
Letting $A\to\infty$, we conclude $u(x,t)\leq f(x)$ for every $(x,t)\in\R^m\times\Interval{0,T}$. 
In particular, the $g(t)$-length of radial curves $\gamma(r)=r\sigma$ emitting from $\sigma\in\Sp^{m-1}\subset \R^m$ is estimated by 
\begin{align*}
\int^{\infty}_{1}\sqrt{u(r\sigma,t)}\,\dd r
\leq \sqrt{b}\int^{\infty}_{1}r^{-2}\,\dd r
= \sqrt{b}<\infty 
\end{align*}
which means that $(\R^m,g(t))$ is geodesically incomplete. 
\end{proof}

\begin{prop}[Lower barrier, rotationally symmetric case]
\label{prop:lower}
Let $(g(t))_{t\in[0,T]}$ be a conformally hyperbolic, instantaneously complete Yamabe flow on $\Hyp$. 
Under the assumption, that $g(t)$ is rotationally symmetric around some point $x_0\in\Hyp$, we have
\begin{align*}
\forall t\in[0,T]:\quad
m(m-1)\,t\,g_{\Hyp}\leq g(t).
\end{align*}
\end{prop}

The proof of Proposition \ref{prop:lower} is based on the following Lemma.

\begin{lem}\label{lem:time-distance}
Let $(g(t))_{t\in[0,T]}$ be an instantaneously complete Yamabe flow on $\Hyp$ given by $g(t)=u(\cdot,t)g_{\Hyp}$ such that $u(\cdot,t)$ depends only on the hyperbolic distance $r$ from some point $x_0\in\Hyp$ for every $t\in\Hyp$. 
Let $\varrho\colon\Hyp\times\intervaL{0,T}\to\Interval{0,\infty}$ be given by 
\begin{align*}
\varrho(r,t)&=\begin{cases}\displaystyle\int^{r}_{a}\sqrt{u(s,t)}\,\dd s & \text{ if } r>a
\vcentcolon=\operatorname{artanh}\bigl(\frac{m-1}{m}\bigr), 
\\
0 & \text{ else  }
\end{cases}
\end{align*}
and let $A\vcentcolon=\max\limits_{t\in[0,T]}\abs[\big]{\frac{\partial}{\partial r}u^{-\frac{1}{2}}(a,t)}$. 
Then, whenever $r>a$, 
\begin{align*}
\biggl(\frac{1}{m-1}\frac{\partial\varrho}{\partial t}
-\Delta_{g}\varrho \biggr)
\geq -m u^{-\frac{1}{2}}
+(m-2)\frac{\partial u^{-\frac{1}{2}}}{\partial r}
-A.
\end{align*}
\end{lem}

\begin{proof}
Given a conformal metric $g=u g_{\Hyp}$ on $\Hyp$ and any smooth function $f\colon\Hyp\to\R$,  
\begin{align}
\Delta_{g}f
&=\frac{1}{\sqrt{\abs{\det{g}}}}
\partial_j\Bigl(\sqrt{\abs{\det{g}}}\,{g}^{ij}\partial_i f\Bigr)
=\frac{1}{u^{\frac{m}{2}}\sqrt{\abs{\det g_{\Hyp}}}}
\partial_j\Bigl(u^{\frac{m}{2}-1}\sqrt{\abs{\det g_{\Hyp}}}\, g_{\Hyp}^{ij}\partial_i f\Bigr)
\notag\\\label{eqn:Laplace-conformal}
&=\frac{1}{u}\Delta_{g_{\Hyp}}f+\frac{m-2}{2u^2}\sk{\nabla u,\nabla f}_{g_{\Hyp}}.
\end{align}
In the case $f=\varrho(\cdot,t)$ and since $\Delta_{g_{\Hyp}}r=\frac{m-1}{\tanh r}$ by Lemma~\ref{lem:Laplacehypdist}, we have
\begin{align*}
\Delta_{g}\varrho
&=\frac{1}{u}\Delta_{g_{\Hyp}}\varrho+\frac{m-2}{2u^2}\sk{\nabla u,\nabla\varrho}_{g_{\Hyp}}
=\frac{1}{u}\frac{\partial^2\varrho}{\partial r^2}+\frac{(m-1)}{u\tanh r}\frac{\partial\varrho}{\partial r}
+\frac{(m-2)}{2u^2}\frac{\partial u}{\partial r}\frac{\partial\varrho}{\partial r}
\end{align*}
where we used the assumption of rotational symmetry. 
If $r>a$, then
\begin{align}\notag
\Delta_{g}\varrho
=\frac{1}{u}\frac{\partial\sqrt{u}}{\partial r}
+\frac{(m-1)}{\sqrt{u}\tanh r}
+\frac{(m-2)}{2u\sqrt{u}}\frac{\partial u}{\partial r} 
&=\frac{(m-1)}{\sqrt{u}\tanh r}
+\frac{(m-1)}{2u\sqrt{u}}\frac{\partial u}{\partial r}
\\\label{eqn:20190208-1}
&\leq m u^{-\frac{1}{2}}
-(m-1)\frac{\partial u^{-\frac{1}{2}}}{\partial r} 
\end{align}
since $(\tanh r)\geq\frac{m-1}{m}$ for $r>a$ by definition of $a$. 
From equation \eqref{eqn:Yamabe-flow} for $u$, we conclude
\begin{align*}
\frac{1}{m-1}\frac{\partial\sqrt{u}}{\partial t}
&=\frac{m}{2\sqrt{u}}
+\frac{\Delta_{g_{\Hyp}}u}{2u\sqrt{u}}
+\frac{(m-6)}{8\sqrt{u}}\frac{\abs{\nabla u}_{g_{\Hyp}}^2}{u^2}.
\end{align*}
Since $\Delta_{g_{\Hyp}} u^{-\frac{1}{2}}
=-\tfrac{1}{2}u^{-\frac{3}{2}}\Delta_{g_{\Hyp}} u
+\tfrac{3}{4}u^{-\frac{5}{2}}\abs{\nabla u}_{g_{\Hyp}}^2$ we have in fact
\begin{align*}
\frac{1}{m-1}\frac{\partial\sqrt{u}}{\partial t}
&=\frac{m}{2\sqrt{u}}
-\Bigl(\Delta_{g_{\Hyp}}\frac{1}{\sqrt{u}}\Bigr)
+\frac{m}{8\sqrt{u}}\frac{\abs{\nabla u}_{g_{\Hyp}}^2}{u^2}. 
\end{align*}
Moreover, with 
\begin{align*}
\frac{m}{8\sqrt{u}}\frac{\abs{\nabla u}_{g_{\Hyp}}^2}{u^2}
&=\frac{m\abs[\big]{2u^{1+\frac{1}{2}}\nabla u^{-\frac{1}{2}}}_{g_{\Hyp}}^2}{8u^{2+\frac{1}{2}}}
=\frac{m}{2}\sqrt{u}\abs[\big]{\nabla u^{-\frac{1}{2}}}_{g_{\Hyp}}^2
\end{align*}
we arrive at 
\begin{align*}
\frac{1}{m-1}\frac{\partial\sqrt{u}}{\partial t}
&=\frac{m}{2\sqrt{u}}
-\Bigl(\Delta_{g_{\Hyp}}\frac{1}{\sqrt{u}}\Bigr)
+\frac{m}{2}\sqrt{u}\abs[\big]{\nabla u^{-\frac{1}{2}}}_{g_{\Hyp}}^2 
\\
&=\frac{m}{2\sqrt{u}}
-\frac{\partial^2 u^{-\frac{1}{2}}}{\partial r^2}
-\frac{m-1}{\tanh r}\frac{\partial u^{-\frac{1}{2}}}{\partial r}
+\frac{m\sqrt{u}}{2}\abs[\bigg]{\frac{\partial u^{-\frac{1}{2}}}{\partial r}}^2.
\end{align*}
For any $b,c>0$ and all $X\in\R$ the inequality
\begin{align}\label{eqn:square}
-b X+c X^2\geq-\frac{b^2}{4c}
\end{align}
follows by completing the square.  
We apply \eqref{eqn:square} with $X=\frac{\partial}{\partial r}u^{-\frac{1}{2}}$, $b=\frac{m-1}{\tanh r}$ and $c=\frac{1}{2}m\sqrt{u}$ to estimate
\begin{align*}
\frac{1}{m-1}\frac{\partial\sqrt{u}}{\partial t}
&\geq\frac{m}{2\sqrt{u}}
-\frac{\partial^2 u^{-\frac{1}{2}}}{\partial r^2}
-\frac{(m-1)^2}{2m(\tanh r)^2\sqrt{u}}
\geq-\frac{\partial^2 u^{-\frac{1}{2}}}{\partial r^2}
\end{align*}
where the last inequality requires $(\tanh r)\geq\frac{m-1}{m}$ which holds by construction whenever $r\geq a$. 
Hence, 
\begin{align}\label{eqn:integrate}
\frac{1}{m-1}\frac{\partial\varrho}{\partial t}(r,t)
&\geq \frac{\partial u^{-\frac{1}{2}}}{\partial r}(a,t)
-\frac{\partial u^{-\frac{1}{2}}}{\partial r}(r,t). 
\end{align}
The claim follows by subtracting \eqref{eqn:20190208-1} from \eqref{eqn:integrate}. 
\end{proof}

\begin{proof}[Proof of Proposition \ref{prop:lower}]
Let $(g(t))_{t\in\Interval{0,T}}$ be an instantaneously complete Yamabe flow on $\Hyp$ given by $g(t)=u(\cdot,t)g_{\Hyp}$ and let $v=\frac{1}{u}$. 
The goal is to prove the uniform estimate $v(\cdot,t)\leq\frac{1}{m(m-1)t}$ for every $t\in\interval{0,T}$. 
Equation \eqref{eqn:Yamabe-flow} implies that $v$ satisfies 
\begin{align}\notag
\frac{-v^{-2}}{m-1}
\frac{\partial v}{\partial t}
&=m+\bigl(2v^{-3}\abs{\nabla v}_{g_{\Hyp}}^2-v^{-2}\Delta_{g_{\Hyp}} v\bigr)v
+\frac{m-6}{4}v^{2}\abs[]{v^{-2}\nabla v}_{g_{\Hyp}}^2
\intertext{and hence evolves by the equation }
\label{eqn:v}
\frac{1}{m-1}\frac{\partial v}{\partial t}
&=-m v^2+v\Delta_{g_{\Hyp}} v-\frac{m+2}{4}\abs{\nabla v}_{g_{\Hyp}}^2.
\end{align}
Let $\eta=\frac{m-2}{4}$ provided that $m\geq3$. 
Applying equation \eqref{eqn:Laplace-conformal} to $g_{\Hyp}=v g$, we have
\begin{align*}
v\Delta_{g_{\Hyp}}v-\frac{m+2}{4}\abs{\nabla v}_{g_{\Hyp}}^2
&=\Delta_{g}v+\Bigl(\frac{m-2}{2}-\frac{m+2}{4}\Bigr) \abs{\nabla v}_{g_{\Hyp}}^2
\\
&=\Delta_{g}v+(\eta-1)\tfrac{1}{v}\abs{\nabla v}_{g}^2
=\tfrac{1}{\eta}v^{1-\eta}\Delta_g v^{\eta}. 
\end{align*}
Hence, equation \eqref{eqn:v} implies
\begin{align}\label{eqn:v-eta}
\frac{1}{m-1}\frac{\partial v^{\eta}}{\partial t}&=-m\eta v^{\eta+1}+\Delta_{g} v^{\eta}
\end{align}	
where we stress the fact, that \eqref{eqn:v-eta} involves the Laplace--Beltrami operator $\Delta_g$ with respect to the time-dependent metric $g(t)$.  

By assumption, there exists a point $x_0\in\Hyp$ such that $u(\cdot,t)$ and $v(\cdot,t)$ depend only on the hyperbolic distance $r$ from $x_0$ for all $t\in[0,T]$. 
Let $\varphi\colon\R\to[0,1]$ be a smooth, nonincreasing cutoff function as in Lemma \ref{lem:cutoffest} satisfying 
$\varphi(s)=1$ if $s\leq 1$ and $\varphi(s)=0$ if $s\geq2$. 
Let $\varrho\colon\Hyp\times\intervaL{0,T}\to\Interval{0,\infty}$ be defined as in Lemma~\ref{lem:time-distance}. 
Given $0<\varepsilon<\frac{1}{4}$, we introduce the functions  
\begin{align}\label{eqn:20190208-4}
\chi  &\vcentcolon=             \varphi  \circ(\varepsilon\varrho), & 
\chi' &\vcentcolon=\varepsilon  \varphi' \circ(\varepsilon\varrho), & 
\chi''&\vcentcolon=\varepsilon^2\varphi''\circ(\varepsilon\varrho)
\end{align}
which are smooth in $\Hyp\times\intervaL{0,T}$ because $\varphi$ is constant around zero. 
We remark that while $\varphi'$ denotes the actual first derivative of the function $\varphi$ of one variable, $\chi'$ is just a convenient shorthand notation.  
In fact, $\abs{\nabla\chi}_{g}^2=\abs{\chi'\nabla\varrho}_{g}^2=\abs{\chi'}^2$ and $\Delta_{g}\chi=\chi''+\chi'\Delta_g\varrho$. 
Hence, 
\begin{align}\notag
&\frac{1}{m-1}\frac{\partial(\chi v^{\eta})}{\partial t}
=\frac{v^{\eta}}{m-1}\frac{\partial\chi}{\partial t}
-m\eta\chi v^{\eta+1}
+\Delta_g(\chi v^{\eta})
-v^{\eta}\Delta_g\chi
-2\sk{\nabla\chi,\nabla v^{\eta}}_g
\\[1ex]
&=\biggl(\frac{1}{m-1}\frac{\partial\varrho}{\partial t}-\Delta_{g}\varrho\biggr)\chi'v^{\eta}
-m\eta\chi v^{\eta+1}
+\Delta_g(\chi v^{\eta})
\label{eqn:20190208-2}
-\chi'' v^{\eta}
-2\sk{\nabla\chi,\nabla v^{\eta}}_g.
\end{align}
Recall that $\chi'\leq0$. 
With the estimate from Lemma \ref{lem:time-distance}, we have 
\begin{align}
\biggl(\frac{1}{m-1}\frac{\partial\varrho}{\partial t}-\Delta_{g}\varrho\biggr)\chi'v^{\eta}
\leq  -m\chi'v^{\eta+\frac{1}{2}}
+(m-2)\chi'v^{\eta}\frac{\partial v^{\frac{1}{2}}}{\partial r}
-A\chi' v^{\eta}.
\end{align}
Surprisingly, since $\frac{\partial}{\partial r}\chi=\chi'\sqrt{u}$, the term 
\begin{align*}
(m-2)\chi'v^{\eta}\frac{\partial v^{\frac{1}{2}}}{\partial r}
&=2\chi'v^{\frac{1}{2}}\frac{\partial v^{\eta}}{\partial r}
=2v\frac{\partial\chi}{\partial r}\frac{\partial v^{\eta}}{\partial r}
=2\sk{\nabla\chi,\nabla v^{\eta}}_g
\end{align*}
cancels the last term in \eqref{eqn:20190208-2}. 
Thus,
\begin{align}
\frac{1}{m-1}\frac{\partial(\chi v^{\eta})}{\partial t}
&\leq -m\chi' v^{\eta+\frac{1}{2}}
-A\chi' v^{\eta}
-m\eta\chi v^{\eta+1}
+\Delta_g(\chi v^{\eta})
-\chi'' v^{\eta}.
\end{align}
By Lemma \ref{lem:cutoffest}, we may choose $\varphi$ such that there exists a constant $C$ depending only on $m$ such that 
\begin{align*}
\abs{m\varphi'}&\leq C\varphi^{\frac{\eta+\frac{1}{2}}{\eta+1}}, &
\abs{\varphi'}+\abs{\varphi''}&\leq C\varphi^{\frac{\eta}{\eta+1}}.
\end{align*}
With this choice, 
\begin{align*}
\frac{1}{m-1}\frac{\partial(\chi v^{\eta})}{\partial t}
&\leq C\varepsilon \chi^{\frac{\eta+\frac{1}{2}}{\eta+1}} v^{\eta+\frac{1}{2}}
+(A+\varepsilon)C\varepsilon\chi^{\frac{\eta}{\eta+1}}v^{\eta}
-m\eta\chi v^{\eta+1}+\Delta_g(\chi v^{\eta}).
\end{align*}
By Young's inequality, $ab\leq \frac{a^p}{p}+\frac{b^q}{q}$ for any $a,b\geq0$ and $p,q>1$ with $\frac{1}{p}+\frac{1}{q}=0$.  
We apply it with $p=2\eta+2$ to estimate
\begin{align*}
C\varepsilon \chi^{\frac{\eta+\frac{1}{2}}{\eta+1}} v^{\eta+\frac{1}{2}}
&\leq C^{2\eta+2}\varepsilon + \varepsilon\chi v^{\eta+1}
\end{align*}
and with $p=\eta+1$ to obtain
\begin{align*}
(A+\varepsilon)C\varepsilon\chi^{\frac{\eta}{\eta+1}}v^{\eta}
&\leq \bigl((A+1)C\bigr)^{\eta+1}\varepsilon + \varepsilon \chi v^{\eta+1}.
\end{align*}
Consequently, introducing the constant $\tilde{C}=C^{2\eta+2}+\bigl((A+1)C\bigr)^{\eta+1}$,
\begin{align}\label{eqn:20190208-3}
\frac{1}{m-1}\frac{\partial(\chi v^{\eta})}{\partial t}
&\leq \tilde{C}\varepsilon
-\bigl(m\eta-2\varepsilon\bigr)\chi v^{\eta+1}
+\Delta_g(\chi v^{\eta}).
\end{align}
Provided that $0<\varepsilon<\frac{1}{4}$, the term involving $\chi v^{\eta+1}$ in \eqref{eqn:20190208-3} has a negative sign. In this case, since $0\leq\chi\leq1$, we may replace \eqref{eqn:20190208-3} by 
\begin{align}
\frac{1}{m-1}\frac{\partial(\chi v^{\eta})}{\partial t}
&\leq \tilde{C}\varepsilon
-\bigl(m\eta-2\varepsilon\bigr)(\chi v^{\eta})^{\frac{\eta+1}{\eta}}
+\Delta_g(\chi v^{\eta}).
\end{align}
The assumption of instantaneous completeness of $g(t)$ implies that $\varrho(r,t)\to\infty$ as $r\to\infty$ for every $t\in\intervaL{0,T}$. 
Therefore, $\chi =\varphi\circ(\varepsilon\varrho)$ is compactly supported in $\Hyp$ for every $t\in\intervaL{0,T}$ and $w\colon\intervaL{0,T}\to\interval{0,\infty}$ given by 
\begin{align*}
w(t)\vcentcolon=\max_{\Hyp}(\chi v^{\eta})(\cdot,t)
\end{align*}
is well-defined. 
Let $t_0\in\intervaL{0,T}$ be arbitrary. 
Let $q_0\in\Hyp$ such that $w(t_0)=\chi v^{\eta}(q_0,t_0)$. 
We compute 
\begin{align*}
&\frac{1}{m-1}\liminf_{\tau\searrow0}\frac{1}{\tau}\biggl(\bigl(w(t_0)\bigr)^{-\frac{1}{\eta}}-\bigl(w(t_0-\tau)\bigr)^{-\frac{1}{\eta}}\biggr) 
\\
&\geq\frac{1}{m-1}\frac{\partial(\chi v^{\eta})^{-\frac{1}{\eta}}}{\partial t} (q_0,t_0)
=-\frac{(\chi v^{\eta})^{-\frac{\eta+1}{\eta}}}{\eta(m-1)}\frac{\partial(\chi v^{\eta})}{\partial t} (q_0,t_0)
\\
&\geq-\frac{1}{\eta}\tilde{C}\varepsilon\bigl(w(t_0)\bigr)^{-\frac{\eta+1}{\eta}}
+\bigl(m-\tfrac{2}{\eta}\varepsilon\bigr)
\end{align*}
where $-\Delta_{g}(\chi v^{\eta})(q_0,t_0)\geq 0$ since $q_0$ is a maximum. 
We conclude that either 
\begin{align*}
\sqrt{\varepsilon}-\tilde{C}\varepsilon \bigl(w(t_0)\bigr)^{-\frac{\eta+1}{\eta}}<0
\end{align*}
which is equivalent to $\bigl(w(t_0)\bigr)^{\frac{\eta+1}{\eta}}<\tilde{C}\sqrt{\varepsilon}$, or 
\begin{align*}
\liminf_{\tau\searrow0}\frac{1}{\tau}\Bigl(\bigl(w(t_0)\bigr)^{-\frac{1}{\eta}}-\bigl(w(t_0-\tau)\bigr)^{-\frac{1}{\eta}}\Bigr) 
&\geq(m-1)\bigl(m-\tfrac{3}{\eta}\sqrt{\varepsilon}\bigr)
\end{align*}
which shows that $w$ is decreasing as long as $w^{\frac{\eta+1}{\eta}}\geq\tilde{C}\sqrt{\varepsilon}$. 
Hence, 
\[
\bigl\{t\in\interval{0,T}\mid\bigl(w(t)\bigr)^{\frac{\eta+1}{\eta}}>\tilde{C}\sqrt{\varepsilon}\bigr\}
=\interval{0,\beta}
\]
for some $\beta\in\interval{0,T}$ because the map $t\mapsto w(t)$ is continuous.  
By Lemma \ref{lem:backdiff}, we have 
\begin{align*}
\bigl(w(t)\bigr)^{-\frac{1}{\eta}}\geq(m-1)\bigl(m-\tfrac{3}{\eta}\sqrt{\varepsilon}\bigr)t
\end{align*}
for every $t\in\interval{0,\beta}$. 
For all $t\in\interval{0,T}$ we may conclude 
\begin{align*}
\bigl(w(t)\bigr)^{\frac{1}{\eta}}\leq \max\left\{\frac{1}{(m-1)(m-\frac{3}{\eta}\sqrt{\varepsilon})t},\bigl(\tilde{C}\sqrt{\varepsilon}\bigr)^{\frac{1}{\eta+1}} \right\}.
\end{align*}
Letting $\varepsilon\searrow0$ such that $\chi\to1$ pointwise in $\Hyp$ proves the claim. 
\end{proof}

\subsection{Generalisation of Topping's interior area estimate}

Topping \cite{Topping2015} proves uniqueness of instantaneously complete Ricci flows on surfaces by estimating differences of area. 
In the following, we adapt his method to the Yamabe flow in dimension $m\geq3$. 
 
The Poincar\'{e} ball model realises hyperbolic space $(\Hyp,g_{\Hyp})$ as the unit ball in $\R^m$ equipped with polar coordinates $P\colon\interval{0,1}\times\Sp^{m-1}\to\Hyp$ mapping $(\rho,\vartheta)\mapsto\rho\vartheta$ and Riemannian metric
\begin{align*}
P^*g_{\Hyp}=
\frac{4}{{(1-\rho^2)}^2}(\dd\rho^2+\rho^{2}\,g_{\Sp^{m-1}}).
\end{align*}
In this section however, logarithmic polar coordinates $\tilde{P}\colon\interval{0,\infty}\times \Sp^{m-1}\to\Hyp$ given by $\tilde{P}(s,\vartheta)=P(\ex^{-s},\vartheta)$ are more suitable. 
We record  
\begin{align}\notag
\tilde{P}^*g_{\Hyp}
&=\frac{4\ex^{-2s}}{{(1-\ex^{-2s})}^2}
\bigl(\dd s^2+ g_{\Sp^{m-1}}\bigr)
=\frac{1}{(\sinh s)^2}
\bigl(\dd s^2+ g_{\Sp^{m-1}}\bigr), 
\\[1ex]\label{eqn:P*gE}
\tilde{P}^*g_{\Eucl}
&=\ex^{-2s}\bigl(\dd s^2+ g_{\Sp^{m-1}}\bigr)
=(\ex^{-s}\sinh s)^2\,\tilde{P}^*g_{\Hyp}
\end{align}
and note that the Riemannian manifold $(\Zl,\zeta)\vcentcolon=\bigl(\interval{0,\infty}\times\Sp^{m-1},\dd s^2+g_{\Sp^{m-1}}\bigr)$ has constant scalar curvature $\Rsc_\zeta\equiv(m-1)(m-2)$.

\begin{proof}[Proof of Theorem \ref{thm:uniqueness}]
Let $g(t)=u(\cdot,t) g_{\Hyp}$ and $\tilde{g}(t)=v(\cdot,t) g_{\Hyp}$ be two Yamabe flows on $\Hyp$ for $t\in[0,T]$. 
Let $U,V\colon\Zl\times[0,T]\to\interval{0,\infty}$ such that $\tilde{P}^*g(t)=U(\cdot,t)\,\zeta$ and $\tilde{P}^*\tilde{g}(t)=V(\cdot,t)\,\zeta$.
From equation \eqref{eqn:Yamabe-flow-eta} follows that $U$ and $V$ both solve 
\begin{align}\label{eqn:20190307-3}
\frac{1}{(\eta+1)(m-1)}
\frac{\partial}{\partial t}U^{\eta+1}
&=-4\eta U^{\eta}
+\frac{1}{\eta}\Delta_{\zeta}U^{\eta},
\end{align}
where $4\eta=(m-2)=\frac{1}{m-1}\Rsc_{\zeta}$ and $\Delta_{\zeta}=\tfrac{\partial^2}{\partial s^2}+\Delta_{g_{\Sp^{m-1}}}$ is the Laplace-Beltrami operator with respect to the metric $\zeta=\dd s^2+g_{\Sp^{m-1}}$ on $\Zl$.  
Note that $U$, $V$ and their derivatives with respect to $s$ have exponential decay for $s\to\infty$. 
In fact, 
\begin{align}
U((s,\theta),t)&=\frac{u(\ex^{-s}\theta,t)}{(\sinh s)^2}, 
\label{eqn:180830-1}
\\[1ex]\notag
\frac{\partial}{\partial s} U^{\eta}((s,\theta),t)
&=\frac{\eta u^{\eta-1}(\ex^{-s}\theta,t)}{(\sinh s)^{2\eta}}\frac{\partial}{\partial s}\bigl(u(\ex^{-s}\theta,t)\bigr)
-2\eta u^{\eta}(\ex^{-s}\theta,t)\frac{\cosh s}{(\sinh s)^{2\eta+1}}
\\\label{eqn:180830-2}
&=\frac{-\eta\,\ex^{-s}}{(\sinh s)^{2\eta}}\Bigl[u^{\eta-1} \frac{\partial u}{\partial \rho}\Bigr] (\ex^{-s}\theta,t)
-\frac{2\eta \, U^{\eta}((s,\theta),t)}{\tanh s}. 
\end{align}
Since $u$ is positive and regular at the origin, $[u^{\eta-1} \frac{\partial u}{\partial \rho}](\ex^{-s}\theta,t)$ stays bounded as $s\to\infty$. 
 
By assumption \ref{cond:Uniqueness-i} and equation \eqref{eqn:P*gE}, applying Proposition \ref{prop:upper} to $\tilde{g}(t)$, we have
\begin{align}\label{eqn:180829}
\bigl((\sinh s)^{2}V\bigr)^{\eta}&\leq \bigl(m(m-1)t\bigr)^{\eta}
+\bigl(b\,(\ex^{-s}\sinh s)^2\bigr)^{\eta}.  
\end{align}
Assumption \ref{cond:Uniqueness-ii} is equivalent to
\begin{align}
\label{est:lowerU}
(\sinh s)^2U&\geq m(m-1)t.
\end{align}
Combining \eqref{eqn:180829} and \eqref{est:lowerU}, we obtain 
\begin{align}
\label{est:V-U1}
V^{\eta}-U^{\eta}
&\leq (b\,\ex^{-s})^{\eta}\leq b^{\eta}. 
\end{align}
Let $S,s_0\in\interval{0,\infty}$ be such that $0<S\leq\frac{1}{3}s_0<s_0\leq\log 2$. 
Let $\varphi\colon\interval{0,\infty}\to[0,1]$ be a cutoff function which is identically equal to $1$ in the interval $\interval{s_0,\infty}$, vanishes in $\interval{0,S}$ and satisfies $\varphi''(s)\leq0$ for $s>\frac{1}{2}s_0$. 
Let $\Zl_S\vcentcolon=\interval{S,\infty}\times\Sp^{m-1}$. 
We analyse the evolution of the quantity $J\colon[0,T]\to\R$ given by 
\begin{align*}
J(t)\vcentcolon=\int_{\Zl_S}\bigl(V^{\eta+1}(\cdot,t)-U^{\eta+1}(\cdot,t)\bigr)_{+}\varphi\,\dd\mu_{\zeta}
\end{align*}
where $x_{+}\vcentcolon=\max\{x,0\}$. 
Abbreviating $w\vcentcolon=V^{\eta+1}-U^{\eta+1}$, we have 
for every fixed $t\in\intervaL{0,T}$ and every $0<\tau<t$
\begin{align*}
J(t)-J(t-\tau)
&=\int_{\Zl_S}w_+(\cdot,t)\,\varphi\,\dd\mu_{\zeta}
 -\int_{\Zl_S}w_+(\cdot,t-\tau)\,\varphi\,\dd\mu_{\zeta}
\\
&\leq\int_{\Zl_S\cap\{w(\cdot,t)>0\}}\bigl(w_+(\cdot,t)-w_+(\cdot,t-\tau)\bigr)\varphi\,\dd\mu_{\zeta}
\\
&\leq\int_{\Zl_S\cap\{w(\cdot,t)>0\}}\bigl(w(\cdot,t)-w(\cdot,t-\tau)\bigr)\varphi\,\dd\mu_{\zeta}.  
\shortintertext{We obtain}
\Psi(t)&\vcentcolon=\limsup_{\tau\searrow0}\frac{J(t)-J(t-\tau)}{\tau}
\\
&\leq \limsup_{\tau\searrow0}\int_{\Zl_S\cap\{w(\cdot,t)>0\}}\frac{1}{\tau}\bigl(w(\cdot,t)-w(\cdot,t-\tau)\bigr)\varphi\,\dd\mu_{\zeta}
\\
&=\int_{\Zl_S\cap\{w(\cdot,t)>0\}}\frac{\partial w}{\partial t}(\cdot,t)\,\varphi\,\dd\mu_{\zeta}
\end{align*}
where we may interchange limit and integral because with $f\vcentcolon=V^{\eta}-U^{\eta}$ and \eqref{eqn:20190307-3} we have
\begin{align}\label{eqn:20190307-5}
\frac{\partial w}{\partial t}=(m-1)(\eta+1)\Bigl(-4\eta\,f+\frac{1}{\eta}\Delta_\zeta f\Bigr)
\end{align}
which is bounded in $Z_S\times[0,T]$ with exponential decay for $s\to\infty$. 
We claim that 
\begin{align}\label{eqn:180830-Green}
\int_{\Zl_S\cap\{f(\cdot,t)>0\}}(\varphi\Delta_{\zeta}f-f\Delta_{\zeta}\varphi)\,\dd\mu_{\zeta}\leq 0.
\end{align}
Indeed, let $(m_k)_{k\in\mathbb{N}}$ be a sequence of regular values for $f(\cdot,t)$ such that $m_k\searrow0$ as $k\to\infty$. 
Then, $\{f(\cdot,t)>m_k\}\subset\Zl$ is a regular, open set with outer unit normal $\nu$ in the direction of $-\nabla f$. 
Moreover, since $f$ and $\nabla f$ have exponential decay for $s\to\infty$ according to \eqref{eqn:180830-1} and \eqref{eqn:180830-2}, since $\varphi(S)=0$ and since $\nabla\varphi$ is supported in $[S,s_0]$, we have by Green's formula
\begin{align*}
\int_{\Zl_S\cap\{f(\cdot,t)>m_k\}}(\varphi\Delta_{\zeta}f-f\Delta_{\zeta}\varphi)\,\dd\mu_{\zeta}
&=\int_{\Zl_S\cap\partial\{f(\cdot,t)>m_k\}}
(\varphi\nabla f\cdot\nu-f\nabla\varphi\cdot\nu)\,\dd\sigma
\\
&\leq 
-m_k\int_{\Zl_S\cap\partial\{f(\cdot,t)>m_k\}}\nabla\varphi\cdot\nu \,\dd\sigma
\\
&=-m_k\int_{\Zl_S\cap\{f(\cdot,t)>m_k\}}\Delta_{\zeta}\varphi\,\dd\mu_{\zeta}
\\
&\leq m_k\int_{\Zl_S\setminus\Zl_{s_0}}\abs{\varphi''}\,\dd\mu_{\zeta}. 
\end{align*}
Passing to the limit $k\to\infty$ proves \eqref{eqn:180830-Green} since $\Zl_S\setminus\Zl_{s_0}$ is a bounded domain. 
Hence, 
\begin{align}\notag
\frac{\Psi}{m-1} 
&\leq \int_{\Zl_S\cap\{w(\cdot,t)>0\}}\frac{1}{m-1}\frac{\partial w}{\partial t}\varphi\,\dd\mu_{\zeta}
\\[1ex]\notag
&=\int_{\Zl_S\cap\{f(\cdot,t)>0\}}
\bigl(-4\eta(\eta+1)f+\tfrac{\eta+1}{\eta}\Delta_{\zeta}f\bigr)\varphi\,\dd\mu_{\zeta}
\\[1ex]\label{eqn:180830-3}
&\leq\int_{\Zl_S\cap\{f(\cdot,t)>0\}} \tfrac{\eta+1}{\eta}f\Delta_{\zeta}\varphi\,\dd\mu_{\zeta}
=\tfrac{\eta+1}{\eta}\int_{\Zl_S}\bigl(V^{\eta}-U^{\eta}\bigr)_+\varphi''\,\dd\mu_{\zeta}.
\end{align} 
Introducing the exponent $\lambda\in\interval{0,\frac{1}{3}}$ we modify estimate \eqref{est:V-U1} as follows.
\begin{align}
\notag
(V^\eta-U^\eta)_+
&\leq b^{(1-\lambda)\eta}{(V^{\eta}-U^{\eta})}_+^{\lambda}
\\\label{est:V-U2}
&\leq b^{(1-\lambda)\eta}
\bigl(U^{-1}\,
(V^{\eta+1}-U^{\eta+1})_+\bigr)^{\lambda}.
\end{align}
We apply estimate \eqref{est:V-U2} and H\"older's inequality with exponents $\frac{1}{\lambda}$ and $\frac{1}{1-\lambda}$ to gain 
\begin{align*}
\frac{\Psi}{m-1}
&\leq C_1\int_{\Sp^{m-1}}\int^{\frac{1}{2}s_0}_{S} 
\bigl(V^{\eta+1}-U^{\eta+1}\bigr)_+^{\lambda}\varphi^{\lambda}
\,\varphi^{-\lambda} U^{-\lambda}
\abs{\varphi''}\,\dd s\,\dd\mu_{g_{\Sp^{m-1}}}
\\
&\leq C_1\,J^{\lambda}
\Bigl(\int_{\Sp^{m-1}}\int^{\frac{1}{2}s_0}_{S}
\varphi^{-\frac{\lambda}{1-\lambda}}
U^{-\frac{\lambda}{1-\lambda}}
\abs{\varphi''}^{\frac{1}{1-\lambda}}
\,\dd s\,\dd\mu_{g_{\Sp^{m-1}}}\Bigr)^{1-\lambda}  
\end{align*}
with constant $C_1=\frac{\eta+1}{\eta}b^{(1-\lambda)\eta}$. 
The restriction of the integration domain to $[S,\frac{1}{2}s_0]$ is justified since $\varphi''\leq0$ in $\interval{\frac{1}{2}s_0,\infty}$. 
Substituting $\lambda=\tfrac{\gamma}{1+\gamma}$ we also have $\tfrac{\lambda}{1-\lambda}=\gamma$ and $\tfrac{1}{1-\lambda}=1+\gamma$ with $\gamma\in\interval{0,\frac{1}{2}}$, and we obtain 
\begin{align}\label{est:dJdt3}
\frac{\Psi}{m-1} 
&\leq C_1\, J ^{\frac{\gamma}{1+\gamma}}
\Bigl(\int_{\Sp^{m-1}}\int^{\frac{1}{2}s_0}_{S}U^{-\gamma}
\abs{\varphi''}^{1+\gamma}\varphi^{-\gamma}\dd s\,\dd\mu_{g_{\Sp^{m-1}}}\Bigr)^{\frac{1}{1+\gamma}}. 
\end{align}
Since $s\mapsto\sinh(s)$ is convex for positive arguments, we have 
\begin{align*}
\forall s\in\interval{0,s_0}\subset\interval{0,\log 2}:\quad
\sinh(s)\leq\sinh(\log2)\frac{s}{\log 2}=\frac{3s}{4\log2}. 
\end{align*} 
The lower barrier \eqref{est:lowerU} then implies 
\begin{align}
\label{est:dJdt4}
\frac{1}{U}\leq\frac{(\sinh s)^2}{m(m-1)t}\leq \frac{C_2}{t} s^2 
\end{align}
with constant $C_2=\frac{1}{m(m-1)}\bigl(\frac{3}{4(\log 2)}\bigr)^2$. 
Substituting \eqref{est:dJdt4} into \eqref{est:dJdt3}, we obtain 
\begin{align*}
\Psi(t)
&\leq C\bigl(J(t)\bigr)^{\frac{\gamma}{1+\gamma}}
\Bigl(\int^{\frac{1}{2}s_0}_{S}
\Bigl(\frac{s^2}{t}\Bigr)^{\gamma}
\abs{\varphi''}^{1+\gamma}\varphi^{-\gamma}\,\dd s\Bigr)^{\frac{1}{1+\gamma}}
\end{align*} 
with constant $C=(m-1)C_1C_2^{\frac{\gamma}{1+\gamma}}\abs{\Sp^{m-1}}^{\frac{1}{1+\gamma}}$, i.\,e. 
\begin{align*}
\limsup_{\tau\searrow0}\frac{J(t)-J(t-\tau)}{\tau}
&\leq  C\Bigl(\frac{J(t)}{t}\Bigr)^{\frac{\gamma}{1+\gamma}}Q^{\frac{1}{1+\gamma}}, 
&
Q&\vcentcolon=\int^{\frac{1}{2}s_0}_{S}s^{2\gamma}\abs{\varphi''}^{1+\gamma}\varphi^{-\gamma}\,\dd s, 
\end{align*}
which resembles Topping's \cite[(3-10)]{Topping2015} ``main differential inequality'' 
\begin{align*}
\tfrac{\dd}{\dd t}J^{\frac{1}{1+\gamma}}\leq C\, t^{-\frac{\gamma}{1+\gamma}}Q^{\frac{1}{1+\gamma}}
\end{align*}
in the $2$-dimensional case. 
By Lemma \ref{lem:backdiff3} stated in the appendix, we conclude  
\begin{align*}
J^{\frac{1}{1+\gamma}}(t)-J^{\frac{1}{1+\gamma}}(0)
\leq C\,t^{\frac{1}{1+\gamma}}Q^{\frac{1}{1+\gamma}}. 
\end{align*}
The assumption $\tilde{g}(0)\leq g(0)$ implies $J(0)=0$. 
Therefore, we have 
\begin{align*}
\int_{\Sp^{m-1}}\int_{s_0}^{\infty}
\bigl(V^{\eta+1}-U^{\eta+1}\bigr)_+\,\dd s\,\dd\mu_{g_{\Sp^{m-1}}}
\leq J(t)\leq C^{1+\gamma} \,t\,Q  
\end{align*}
for every $t\in[0,T]$. 
With a clever choice of $\varphi$, Topping \cite[Prop. 3.2]{Topping2015} proves  
\begin{align}\label{est:Q}
Q\leq \frac{C_\gamma}{s_0}{(\log s_0-\log S)}^{-\gamma}.
\end{align} 
In the limit $S\searrow 0$ we have $Q\to0$ by \eqref{est:Q} which yields $V\leq U$ in $\interval{s_0,\infty}\times\Sp^{m-1}$. 
Since $s_0>0$ is arbitrary, we obtain $V\leq U$ globally and hence $\tilde{g}(t)\leq g(t)$ as claimed.  

In the case that $g(t)$ and $\tilde{g}(t)$ both satisfy \ref{cond:Uniqueness-ii} and $g(0)=\tilde{g}(0)\leq b\,g_{\Eucl}$, 
the reverse inequality $\tilde{g}(t)\geq g(t)$ follows similarly by switching the roles of $U$ and $V$. 
\end{proof}

%===== APPENDIX ========================================
\clearpage
\appendix
\section{Auxiliary results} 

The previous sections depend on some standard results and computations which we collect in this appendix for convenience of the reader. 

\begin{prop}[Inverse function theorem for Banach spaces]\label{prop:IFT}
Let $X,Y$ be Banach spaces and $\tilde{X}\subset X$ be open. 
Let $S\colon\tilde{X}\to Y$ be Fr\'{e}chet differentiable at $x_0$ with $S(x_0)=0$ and invertible derivative $D S(x_0)$. 
Then there exists a neighbourhood $U\subset\tilde{X}$ of $x_0$ such that 
\begin{enumerate}[label={\normalfont(\roman*)}]
	\item $V=SU$ is open,
	\item $S|_U\colon U\to V$ is a homeomorphism,
	\item $(S|_U)^{-1}$ is Fr\'{e}chet differentiable.
\end{enumerate}
\end{prop}

In the following proposition, we denote partial derivatives by subscripts and understand a sum $\sum^{m}_{i=1}$ whenever an index $i$ appears twice in an expression. 

\begin{prop}[Linear parabolic maximum principle {\cite[\textsection\,3.3]{Protter1984}}]\label{prop:parabolicMP}
Let $\Omega\subset\R^m$ be an open, bounded domain. 
Suppose, $u$ satisfies 
\begin{align*} 
\left\{
\begin{aligned}
u_t-a_{ij}u_{x_i x_j}-b_k u_{x_k}-c u&\leq0
&&\text{ in ${\Omega}\times[0,T]$, } \\
u&\leq 0
&&\text{on $\Omega\times\{0\}$ and on $\partial\Omega\times[0,T]$, } 
\end{aligned}\right.
\end{align*}
where the function $c<\lambda\in\mathbb{R}$ is bounded from above and ellipticity $a_{ij}\,\xi_i\,\xi_j\geq0$ for all $\xi\in\R^m$ holds uniformly.
Then, $u\leq0$ in $\Omega\times[0,T]$. 
\end{prop}

\begin{proof}
The function $v(x,t)=u(x,t)\ex^{-\lambda t}$ satisfies the equation 
\begin{align}\label{eqn:cpluslambda}
v_t-a_{ij}v_{x_i x_j}-b_k v_{x_k}-(c-\lambda)v\leq0. 
\end{align}
Assume that $v(x_0,t_0)=\max_{\overline{\Omega}\times[0,T]} v$. 
If $t_0=0$ or if $x_0\in\partial\Omega$, then $v\leq0$ follows. 
If $(x_0,t_0)\in\Omega\times\intervaL{0,T}$, then $v_t(x_0,t_0)\geq0$, $v_{x_k}(x_0,t_0)=0$ and $-a_{ij}v_{x_i x_j}(x_0,t_0)\geq0$. 
Thus, \eqref{eqn:cpluslambda} implies $-(c-\lambda)v(x_0,t_0)\leq 0$.
From $(c-\lambda)<0$ follows $v(x_0,t_0)\leq 0$. 
Therefore, $v\leq 0$ in $\overline{\Omega}\times[0,T]$. 
Since $u$ and $v$ share the same sign, the claim follows. 
\end{proof}

\begin{lem}[Laplacian of hyperbolic distance]\label{lem:Laplacehypdist} 
Let $r\colon\Hyp\to\Interval{0,\infty}$ be the Riemannian 
distance function with respect to some origin in hyperbolic space $(\Hyp,g_{\Hyp})$ of dimension $m\geq2$. Then, 
\begin{align*}
\Delta_{g_{\Hyp}}r&=\frac{m-1}{\tanh r}.
\end{align*}
In particular, $(\Delta_{g_\Hyp}r)(x)\leq 2(m-1)$ for every $x\in\Hyp$ with $r(x)\geq1$. 
\end{lem}

\begin{proof}
The Poincar\'{e} ball model realises hyperbolic space $(\Hyp,g_{\Hyp})$ as the unit ball equipped with polar coordinates $P\colon\interval{0,1}\times\Sp^{m-1}\to\Hyp$ mapping $(\rho,\vartheta)\mapsto\rho\vartheta$ and conformal Riemannian metric $P^*g_{\Hyp}=h^2g_{\Eucl}$, where 
\begin{align}\label{eqn:Poincare-metric}
h(\rho)&=\frac{2}{1-\rho^2}, &
g_{\Eucl}&=\dd\rho^2+\rho^{2}\,g_{\Sp^{m-1}}. 
\end{align}
%(see \cite[§\,22.2]{Postnikov2001}).
Here, $g_{\Sp^{m-1}}$ is the standard metric on the unit sphere $\Sp^{m-1}\subset\R^m$ and $g_{\Eucl}$ is the Euclidean metric on the unit ball in $\R^m$. 
We denote the radial coordinate on the unit ball by $\rho$ and reserve $r$ for the hyperbolic distance which is given by 
\begin{align}\label{eqn:hyperbolic_distance}
r(\rho)
=\int^{\rho}_{0}h(x)\,\dd x
=2\operatorname{artanh}(\rho)
=\log\biggl(\frac{1+\rho}{1-\rho}\biggr).
\end{align}
By \eqref{eqn:hyperbolic_distance}, we have 
\begin{align*}
\frac{\partial r}{\partial\rho}&=h, &
\frac{\partial^2 r}{\partial\rho^2}&=\frac{\partial h}{\partial\rho}
=\rho\,h^2.
\end{align*}
On the one hand, equation \eqref{eqn:Laplace-conformal} implies 
\begin{align}\notag
\Delta_{g_{\Hyp}}r
&=h^{-2}\Delta_{g_{\Eucl}}r+(m-2)h^{-3}\sk{\nabla h,\nabla r}_{g_{\Eucl}}
\\[.5ex]\notag
&=\frac{1}{h^2}\frac{\partial^2 r}{\partial\rho^2}
+\frac{(m-1)}{h^2\rho}\frac{\partial r}{\partial\rho}
+(m-2)\frac{\rho}{h}\frac{\partial r}{\partial\rho}
\\\notag
&=\rho +\frac{(m-1)}{h\rho}+(m-2)\rho
\\
\label{eqn:20190204-4}
&=(m-1)\biggl(\rho+\frac{1-\rho^2}{2\rho}\biggr)
=(m-1)\biggl(\frac{1+\rho^2}{2\rho}\biggr).
\end{align}
On the other hand, equation \eqref{eqn:hyperbolic_distance} implies 
\[ 
\tanh(r)
=\frac{\ex^{r}-\ex^{-r}}{\ex^{r}+\ex^{-r}}
=\frac{\frac{1+\rho}{1-\rho}-\frac{1-\rho}{1+\rho}}{\frac{1+\rho}{1-\rho}+\frac{1-\rho}{1+\rho}}
=\frac{(1+\rho)^2-(1-\rho)^2}{(1+\rho)^2+(1-\rho)^2}
=\frac{4\rho}{2+2\rho^2}.
\] 
Combined with \eqref{eqn:20190204-4}, the claim follows. 
\end{proof}

\begin{lem}[cutoff function]
\label{lem:cutoffest}
Let $\varepsilon>0$ and $a,b>0$ be real parameters. 
Then there exists a non-increasing cutoff function $\varphi\in\Ct^{2}(\mathbb{R})$ given by 
\begin{align*}
\varphi(x)&=\begin{cases}1,&\text{ if }x\leq1, \\
0,&\text{ if }x\geq2   
\end{cases}
\end{align*}
which satisfies the inequality 
\begin{align*}
\abs{\varphi''}+a\abs{\varphi'}+b\frac{\varphi'^2}{\varphi}&\leq C\varphi^{1-\varepsilon}
\end{align*}
in $\{x\in\mathbb{R}\mid \varphi(x)\neq0\}$
with a constant $C$ depending only on $a,b$ and $\varepsilon$. 
\end{lem}

\begin{proof}
There exists a non-increasing function $\psi\in\Ct^{2}(\mathbb{R})$ satisfying $\psi(x)=1$ for $x\leq1$ and $\psi(x)=0$ for $x\geq2$ as well as $\abs{\psi},\abs{\psi'},\abs{\psi''}\leq C_0$. 
Let $p=\frac{2}{\varepsilon}$ and $\varphi=\psi^p$. 
We may assume $0<\delta\leq2$ and $p\geq1$. 
Then, 
\begin{align*}
\varphi'&=p\psi'\psi^{p-1}, 
&
\varphi''&=p\bigl((p-1)\psi'^2+\psi\psi''\bigr)\psi^{p-2}, 
&
\frac{\varphi'^2}{\varphi}&=p^2\psi'^2\psi^{p-2}, 
\end{align*}
which implies 
\begin{align*}
\abs{\varphi''}+a\abs{\varphi'}+b\frac{\varphi'^2}{\varphi}
&\leq p\bigl((p-1)\psi'^2+\abs{\psi\psi''}+a\abs{\psi'\psi}+b p\,\psi'^2\bigr)\psi^{p-2}
\\
&\leq p\bigl(1+a+(1+b)p\bigr)C_0^2\varphi^{\frac{p-2}{p}}
\leq\frac{4(2+a+b)C_0^2}{\varepsilon^2}\,\varphi^{1-\varepsilon}
.\qedhere
\end{align*}
\end{proof}

\begin{lem}\label{lem:backdiff}
Let $Q\in\R$ and let $f\colon [0,T]\to\R$ be continuous satisfying 
\begin{align*}
\liminf_{\tau\searrow0}\frac{f(\xi)-f(\xi-\tau)}{\tau}\geq Q
\end{align*}
for every $0<\xi\leq T$. Then $f(t)-f(0)\geq Qt$ for every $t\in[0,T]$.
\end{lem}

\begin{proof}
We follow an argument by Richard Hamilton \cite[Lemma 3.1]{Hamilton1986}. 
The function 
\begin{align*}
w(t)&\vcentcolon=f(t)-f(0)-Qt 
\end{align*}
is continuous and satisfies 
\begin{align}\label{eqn:proofA5-1}
\liminf_{\tau\searrow0}\frac{w(\xi)-w(\xi-\tau)}{\tau}
&=\liminf_{\tau\searrow0}\frac{f(\xi)-f(\xi-\tau)}{\tau}
-Q\geq 0. 
\end{align}
Let $0<t\leq T$ and $\varepsilon>0$ be arbitrary but fixed. 
Then there exists $\delta>0$ such that  
\begin{align}\label{eqn:proofA5-2}
\forall \tau\in\Interval{0,\delta}:\quad
w(t)-w(t-\tau)\geq-\varepsilon\tau
\end{align}
because otherwise, we would obtain a contradiction to \eqref{eqn:proofA5-1}. 
We may assume that $\delta\in\intervaL{0,t}$ is the largest value such that \eqref{eqn:proofA5-2} holds. 
By continuity of $w$, \eqref{eqn:proofA5-2} implies 
\begin{align}\label{eqn:proofA5-4}
w(t)-w(t-\delta)\geq-\varepsilon\delta. 
\end{align}
If $t-\delta>0$, we may repeat the argument to find some $\delta'>0$ such that  
\begin{align}\label{eqn:proofA5-3}
\forall \tau\in\Interval{0,\delta'}:\quad
w(t-\delta)-w(t-\delta-\tau)\geq-\varepsilon\tau.
\end{align}
In particular, \eqref{eqn:proofA5-4} and \eqref{eqn:proofA5-3} can be combined to 
\begin{align*}
w(t)-w(t-\delta-\tau)
&\geq
%-\varepsilon\delta+w(t-\delta)-\varepsilon\tau-w(t-\delta)=
-\varepsilon(\delta+\tau)
\end{align*}
for all $\tau\in\Interval{0,\delta'}$ in contradiction to the maximality of $\delta$. 
Hence, $\delta=t$ and we obtain
\begin{align*}
w(t)-w(0)\geq-\varepsilon t. 
\end{align*} 
Since $\varepsilon>0$ is arbitrary and $w(0)=0$ we have $w(t)\geq0$ which proves the claim.  
\end{proof}

\begin{lem}\label{lem:backdiff2} 
Let $\eta>0$ and $Q\geq0$. 
Let $v\colon [0,T]\to\Interval{0,\infty}$ be continuous satisfying 
\begin{align*}
\limsup_{\tau\searrow0}\frac{v(\xi)-v(\xi-\tau)}{\tau}\leq\eta Q v(\xi)^{1-\frac{1}{\eta}}
\end{align*}
for every $0<\xi\leq T$ where $v(\xi)\neq0$. 
Then $v(t)^{\frac{1}{\eta}}-v(0)^{\frac{1}{\eta}}\leq Q t$ for every $t\in[0,T]$.
\end{lem}

\begin{proof}
We consider the function $f=v^{\frac{1}{\eta}}$. 
Let $0<\xi\leq T$ be arbitrary. 
If $v(\xi)\neq0$, then by assumption
\begin{align}\notag
\eta Q\geq f(\xi)^{1-\eta}&\limsup_{\tau\searrow0}\frac{1}{\tau}
\bigl(f(\xi)^\eta-f(\xi-\tau)^\eta\bigr)
\\\notag
={}&\limsup_{\tau\searrow0}\frac{1}{\tau} 
\bigl(f(\xi)-f(\xi-\tau)\bigr)\frac{f(\xi)^{\eta}-f(\xi-\tau)^{\eta}}{f(\xi)-f(\xi-\tau)} f(\xi)^{1-\eta}
\\ 
={}&\limsup_{\tau\searrow0}\frac{\eta}{\tau} 
\bigl(f(\xi)-f(\xi-\tau)\bigr). 
\label{eqn:180830-4}
\end{align}
Dividing by $\eta$ yields
\begin{align}\label{eqn:leqQ}
\limsup_{\tau\searrow0}\frac{f(\xi)-f(\xi-\tau)}{\tau}\leq Q. 
\end{align}
If $v(\xi)=0$ we have
\begin{align*}
\limsup_{\tau\searrow0}\frac{f(\xi)-f(\xi-\tau)}{\tau}\leq 0
\end{align*} 
since $f(\xi)=v(\xi)^{\frac{1}{\eta}}=0$ and $f\geq0$. 
We conclude that \eqref{eqn:leqQ} holds not only where $v(\eta)\neq0$ but in fact for every $0<\xi\leq T$ because $Q\geq0$ is assumed. 
The claim follows by multiplying \eqref{eqn:leqQ} with $-1$ and applying Lemma \ref{lem:backdiff}. 
\end{proof}

\begin{lem}\label{lem:backdiff3} 
Let $\alpha\in\interval{0,1}$ and $C\geq0$. 
Let $J\colon [0,T]\to\Interval{0,\infty}$ be continuous satisfying 
\begin{align*}
\limsup_{\tau\searrow0}\frac{J(t)-J(t-\tau)}{\tau}\leq C \Bigl(\frac{J(t)}{t}\Bigr)^{1-\alpha}
\end{align*}
for every $0<t\leq T$. 
Then $J(t)^{\alpha}-J(0)^{\alpha}\leq C t^{\alpha}$ for every $t\in[0,T]$.
\end{lem}

\begin{proof}
As in the proof of Lemma \ref{lem:backdiff2} (with $\alpha=\frac{1}{\eta}$) we obtain 
\begin{align*}
\limsup_{\tau\searrow0}\frac{J(t)^{\alpha}-J(t-\tau)^{\alpha}}{\tau}&\leq\alpha C t^{\alpha-1}  
\end{align*}
for every fixed $t\in\intervaL{0,T}$. 
In particular, $f(\xi)=J(\xi)^{\alpha}-C\xi^{\alpha}$ satisfies 
\begin{align*}
\limsup_{\tau\searrow0}\frac{f(\xi)-f(\xi-\tau)}{\tau}&\leq\alpha C\xi^{\alpha-1}
-C\lim_{\tau\searrow0}\frac{\xi^{\alpha}-(\xi-\tau)^{\alpha}}{\tau}=0
\end{align*}
and we may conclude as before by multiplying with $-1$ and applying Lemma \ref{lem:backdiff}. 
\end{proof}

%===== BIBLIOGRAPHY ====================================

\clearpage 

\def\cprime{$'$}

\end{document}